\documentclass[a4paper,
	11pt,
	pdftex,
	headings=normal,
	headsepline,
	footsepline,
	onecolumn,
	headinclude,
	footinclude,
	DIV14,
	abstracton
]{scrartcl}

\usepackage{
	graphicx,
	amssymb,
	amsmath,
	amsthm,
	xcolor,
	dsfont,
	algpseudocode,
	authblk,
}

\usepackage[ngerman, english]{babel}

\usepackage[bf]{caption}
\usepackage[colorlinks,
	pdffitwindow=false,
	plainpages=false,
	pdfpagelabels=true,
	pdfpagemode=UseOutlines,
	pdfpagelayout=SinglePage,
	bookmarks=false,
	colorlinks=true,
	hyperfootnotes=false,
	linkcolor=blue,
	citecolor=green!50!black
]{hyperref}

\usepackage{enumerate}
\usepackage{algorithm}
\usepackage{subcaption}    
\usepackage{tabularx}
\usepackage[left=2.50cm, right=2.50cm, top=2.00cm, bottom=3.50cm]{geometry}

\newtheorem{theorem}{Theorem}[section]

\newtheorem{definition}[theorem]{Definition}
\newtheorem{remark}[theorem]{Remark}
\newtheorem{example}[theorem]{Example}

\DeclareMathOperator*{\argmin}{arg\,min} 
\newcommand{\R}{\mathbb{R}}
\newcommand{\N}{\mathbb{N}}

\newcommand{\B}{\mathcal{B}}
\newcommand{\D}{\mathcal{D}}
\renewcommand{\L}{\mathcal{L}}
\newcommand{\M}{\mathcal{M}}
\newcommand{\F}{\mathcal{F}}

\newcommand{\vspan}{\text{span}}

\let\OLDthebibliography\thebibliography
\renewcommand\thebibliography[1]{
	\OLDthebibliography{#1}
	\setlength{\parskip}{0pt}
	\setlength{\itemsep}{0pt plus 0.3ex}
}

\begin{document}
	\title{Inverse multiobjective optimization: Inferring decision criteria from data}
	\author[]{Bennet Gebken}
	\author[]{Sebastian Peitz}
	\affil[]{\normalsize Department of Mathematics, Paderborn University, Germany}

	\maketitle

	\begin{abstract}
		It is a very challenging task to identify the objectives on which a certain decision was based, in particular if several, potentially conflicting criteria are equally important and a continuous set of optimal compromise decisions exists. This task can be understood as the inverse problem of multiobjective optimization, where the goal is to find the objective vector of a given Pareto set. To this end, we present a method to construct the objective vector of a multiobjective optimization problem (MOP) such that the Pareto critical set contains a given set of data points or decision vectors. The key idea is to consider the objective vector in the multiobjective KKT conditions as variable and then search for the objectives that minimize the Euclidean norm of the resulting system of equations. By expressing the objectives in a finite-dimensional basis, we transform this problem into a homogeneous, linear system of equations that can be solved efficiently. 
		There are many important potential applications of this approach. Besides the identification of objectives (both from clean and noisy data), the method can be used for the construction of surrogate models for expensive MOPs, which yields significant speed-ups. Both applications are illustrated using several examples.
	\end{abstract}

	\section{Introduction}
	When applying optimization to real-world problems, there are often multiple quantities that have to be optimized at the same time. In production for example, typical goals are the maximization of the quality of a product and the minimization of the production cost. When the objectives are conflicting, there cannot be a single solution that is optimal for all objectives at the same time. This is called a \emph{multiobjective optimization problem (MOP)}. To solve a problem like this, we search for the set of all optimal compromises, the so-called \emph{Pareto set}, containing all \emph{Pareto optimal points}. A point $x^*$ is called Pareto optimal if there exists no other point that is at least as good as $x^*$ in all objectives, but strictly better than $x^*$ in at least one objective.	
	
	While most of the research in multiobjective optimization is concerned with efficiently computing the Pareto set of a given MOP, we here address the \emph{inverse problem of multiobjective optimization}: 
	\begin{equation*}
		\text{Given a set $P \subseteq \R^n$, identify the objectives for which $P$ is the Pareto set.}
	\end{equation*}			
	Although it is possible to state this problem in such a general form, it will have many degenerate solutions that we are are not interested in, since there is no restriction on any type of regularity of the objective functions. Therefore, we will instead consider a more well-behaved version of this problem that arises by using the concept of \emph{Pareto criticality}. A point $x^* \in \R^n$ is called \emph{Pareto critical} if it satisfies the \emph{Karush-Kuhn-Tucker (KKT) conditions} \cite{KT1951}, i.e., if there is a convex combination of all the gradients of the objective functions $f_i \in C^1(\R^n,\R)$, $i \in \{ 1,...,k \}$, in $x^*$ which is zero. In that case, if $\alpha^* \in \R^k$ contains the coefficients of this convex combination, then $\alpha^*$ is called a \emph{KKT vector of} $x^*$ and the pair $(x^*,\alpha^*)$ is called an \emph{extended Pareto critical point}. The set of all such pairs is called the \emph{extended Pareto critical set}.  The above problem can be restated using this concept:
	\begin{equation}\label{eq:IMOP} \tag{IMOP}
		\begin{aligned} 
			&\text{Given a finite \emph{data set} $\D = (\D_x,\D_\alpha) \subseteq \R^n \times \R^k$, find an objective vector} \\
			&\text{$f \in C^1(\R^n,\R^k)$ whose extended Pareto critical set contains $\D$.}
		\end{aligned}
	\end{equation}	
	Since the search space $C^1(\R^n,\R^k)$ is infinite-dimensional, we will consider finite-dimensional subspaces of $C^1(\R^n,\R)$ that are spanned by sets of basis functions $\B \subseteq C^1(\R^n,\R)$. This will transform \eqref{eq:IMOP} into a homogeneous linear system in the coefficients of the basis functions which can be solved by singular value decomposition.
	
	For the single objective case, i.e., for $k = 1$, the problem \eqref{eq:IMOP} is addressed within the field of \emph{inverse optimization}. For combinatorial problems, a survey on inverse optimization can be found in \cite{H2004}. In \cite{HL1999} and \cite{AO2001}, inverse linear problems of the form $\min_x c^\top x$ (with some linear constraints) were considered, where the goal is to find the cost vector $c$ so that a given feasible point is optimal. In \cite{KWB2011}, convex parameter-dependent problems were considered with the intention of estimating the objective functions from observations of parameter values and associated optimal solutions. Recently, first results in the multiobjective case have appeared. In \cite{CCLS2014}, inverse linear multiobjective problems are treated similarly to \cite{KWB2011} by transforming the multiobjective problem into a scalar problem via the Weighting Method. In \cite{CL2017}, the ideas of \cite{KWB2011} and \cite{CCLS2014} are combined with the additional focus on preserving the trade-off in the given solution.
	
	In all previous approaches, certain properties have to be assumed for the objective functions such as linearity, convexity or even a parameter dependent formulation. In contrast to this, we will make no assumptions besides differentiability. Additionally, instead of single points, the approach can be applied to an arbitrary amount of data points. This will allow us to consider the inverse problem of multiobjective optimization in a much more general setting. 
	
	The remainder of this article is structured as follows. We begin with a brief introduction to multiobjective optimization in Section \ref{sec:introduction_MOP} before presenting our main theoretical results in Section \ref{sec:generating_mops_from_data}. There, we will first investigate the existence of an objective vector in the span of the chosen basis $\B$ for which the data points are extended Pareto critical. We then address the task of finding the objective vector whose extended Pareto critical set is as close to a given data set as possible. The application of the resulting algorithm to two important problem classes is presented in Sections \ref{sec:constructingObjectives} and \ref{sec:surrogates}. These are the construction of objective functions from both clean and noisy decision data as well as the generation of surrogate models for expensive MOPs. Finally, we draw a conclusion and discuss possible future work in Section \ref{sec:conclusion}.	
	
	For our numerical results, we use the built-in method \verb+svd+ from MATLAB 2017a for singular value decomposition. For the computation of the extended Pareto critical sets in this article, we use the Continuation Method \verb+CONT-Recover+ from \cite{SDD2005}.
	
	\section{Multiobjective optimization} \label{sec:introduction_MOP}
	In this section, we will briefly introduce the basic concepts of multiobjective optimization. For a more detailed introduction, we refer to \cite{M1998,E2005,H2001}.
	
	Let $f : \R^n \rightarrow \R^k$ be a vector-valued function, called the \emph{objective vector}, with continuously differentiable components $f_i : \R^n \rightarrow \R$, $i \in \{ 1,...,k \}$, called \emph{objective functions}. It maps the \emph{variable space} $\R^n$ to the \emph{image space} $\R^k$. The goal of multiobjective optimization is to minimize the objective vector $f$, i.e., to minimize all objective functions $f_i$ simultaneously. This is called a \emph{multiobjective optimization problem (MOP)} and is denoted by 
	\begin{equation} \label{eq:MOP}
		\min_{x \in \R^n} f(x) = \min_{x \in \R^n} 
		\begin{pmatrix}
			f_1(x) \\
			\vdots \\
			f_k(x)
		\end{pmatrix}. \tag{MOP}
	\end{equation}
	In contrast to scalar optimization (i.e., $k = 1$), it is not immediately clear what we mean by minimizing $f$, as there is no natural total order of the objective values in $\R^k$ for $k > 1$. As a result, we cannot expect to find a single point that solves \eqref{eq:MOP}. Instead, we search for the \emph{Pareto set} which is defined in the following way:
	\begin{definition}
		\begin{itemize}
			\item[(a)] $x^* \in \R^n$ \emph{dominates} $x \in \R^n$, if $f_i(x^*) \leq f_i(x)$ for all $i \in \{1,...,k\}$ and $f_j(x^*) < f_j(x)$ for some $j \in \{1,...,k\}$.
			\item[(b)] $x^* \in \R^n$ is called \emph{locally Pareto optimal} if there exists an open set $U \subseteq \R^n$ with $x^* \in U$ such that there is no point $x \in U$ dominating $x^*$. If this holds for $U = \R^n$, then $x^*$ is called \emph{Pareto optimal}.
			\item[(c)] The set of all (locally) Pareto optimal points is called the \emph{(local) Pareto set}, its image under $f$ the \emph{(local) Pareto front}.  
		\end{itemize}
	\end{definition}
	
	Similar to scalar optimization, there are necessary conditions for local Pareto optimality using the first order derivatives of $f$, called the \emph{Karush-Kuhn-Tucker (KKT)} conditions \cite{H2001}:

	\begin{theorem}
		Let $x^*$ be a locally Pareto optimal point of \eqref{eq:MOP} and
		\begin{equation} \label{eq:Delta_k}
			\Delta_k := \left\{ \alpha \in (\R^{\geq 0})^k : \sum_{i = 1}^k \alpha_i = 1 \right\}.
		\end{equation}
		Then there exists some $\alpha^* \in \Delta_k$ such that
		\begin{equation} \label{eq:KKT}
			Df(x^*)^\top \alpha^* = \sum_{i = 1}^k \alpha^*_i \nabla f_i(x^*) = 0. \tag{KKT}
		\end{equation}
	\end{theorem}
	
	For $k = 1$, these conditions reduce to the well-known optimality condition $\nabla f(x^*) = 0$. The set of points satisfying the KKT conditions is a superset of the (local) Pareto set and we make the following definition:
	
	\begin{definition}
		Let $x \in \R^n$. 
		\begin{itemize}
			\item[a)] If there exists some $\alpha \in \Delta_k$ (with $\Delta_k$ as in \eqref{eq:Delta_k}) such that \eqref{eq:KKT} holds, then $x$ is called \emph{Pareto critical} and $\alpha$ a \emph{KKT vector of} $x$ containing the \emph{KKT multipliers} $\alpha_i$, $i \in \{1,...,k\}$. The set of Pareto critical points $P_c$ of \eqref{eq:MOP} is called the \emph{Pareto critical set}. (Pareto critical points are sometimes also referred to as \emph{substationary points} by other authors.)
			\item[b)] In the situation of a), the pair $(x,\alpha) \in \R^n \times \Delta_k$ is called an \emph{extended Pareto critical point}. The set of all such pairs $P_\M \subseteq \R^n \times \Delta_k$ is called the \emph{extended Pareto critical set}.
		\end{itemize}
	\end{definition}
	
	Since the structure of $P_c$ and $P_\M$ will be important for our approach, we will briefly mention three results: In \cite{dM1976,LP2014} it was shown that $P_c$ is generically a stratification, which basically means that it is a ``manifold with boundaries and corners''. In \cite{GPD2019} it was shown that the boundary (or \emph{edge}) of $P_c$ is covered by Pareto critical sets of subproblems where only subsets of the set of objective functions are optimized. In \cite{HL1999} it was shown that $\{ (x,\alpha) \in P_\M : \alpha \in (\R^{> 0})^k \} \subseteq P_\M$ is a $(k-1)$-dimensional submanifold of $\R^{n+k}$ if a certain rank condition holds. 
	
	\section{Inferring objective vectors from data} \label{sec:generating_mops_from_data}
	We will now present a way to construct objective vectors for which $P_\M$ contains a finite set of given data points $\D_x = \{\bar{x}^1,...,\bar{x}^N \} \subseteq \R^n$ with corresponding KKT vectors $\D_\alpha  = \{ \bar{\alpha}^1, ..., \bar{\alpha}^N \} \subseteq \Delta_k$. The general concept of this inverse approach is to consider $x^*$ and $\alpha^*$ as given in \eqref{eq:KKT} instead of the objective vector $f$. So in contrast to the usual task of searching for an $x \in \R^n$ for which an $\alpha \in \Delta_k$ exists so that \eqref{eq:KKT} holds, we now search for an $f \in C^1(\R^n,\R^k)$ for which \eqref{eq:KKT} holds for all $\bar{x}^j$ and $\bar{\alpha}^j$, $j \in \{ 1, ..., N \}$. As it is infinite-dimensional, we obviously cannot search the entire $C^1(\R^n,\R^k)$. Instead, we consider finite-dimensional linear subspaces of $C^1(\R^n,\R)$ that are spanned by a set of basis functions $\B = \{ b_1, ..., b_d \} \subseteq C^1(\R^n,\R)$, and then search for $f \in \vspan(\B)^k$. An example for the choice of basis functions are the monomials in $n$ variables (up to a certain degree) such that $\vspan(\B)$ is the space of polynomials. The usage of basis functions reduces the task of finding an $f \in C^1(\R^n,\R^k)$ to the task of finding the coefficients $c \in \R^d$ of the corresponding linear combination of basis functions. This problem can be stated as a homogeneous linear problem in $c$ and can be solved efficiently via singular value decomposition. In particular, the smallest singular value can be used as a measure of how well the given data set can be represented as an extended Pareto critical set of an objective vector consisting of the given basis functions.
	
	We will assume for the remainder of this section that the following are given:
	\begin{itemize}
		\item a data set $\D = \{ (\bar{x}^1,\bar{\alpha}^1), (\bar{x}^2,\bar{\alpha}^2), ..., (\bar{x}^N,\bar{\alpha}^N) \} \subseteq \R^n \times \Delta_k$ (and in particular the number of objective functions $k$),
		\item a set of basis functions $\B = \{b_1,...,b_d\} \subseteq C^1(\R^n,\R)$ with linearly independent derivatives.
	\end{itemize}
	
	\subsection{Existence of exact approximations}
	In this subsection, our goal is to find an objective vector with components in $\vspan(\B)$ for which the set $\D$ is exactly extended Pareto critical. In other words, our goal is to find a function $f : \R^n \rightarrow \R^k$, $f = (f_i)_{i \in \{1,...,k\}}$, $f \neq 0$ with $f_i \in \vspan(\B)$ $\forall i \in \{1,...,k\}$ and 
	\begin{equation} \label{eq:KKT_exact}
		Df(\bar{x})^\top \bar{\alpha} = 0 \quad \forall (\bar{x},\bar{\alpha}) \in \D.
	\end{equation}
	To this end, for $f_i \in \vspan(\B)$, we can write 
	\begin{equation} \label{eq:f_in_basis}
		f_i(x) = \sum_{j = 1}^d c_{ij} b_j(x)
	\end{equation}		
	for some $c_i \in \R^d$. Thus, we obtain
	\begin{equation*}
		Df(x)^\top \alpha = \sum_{i = 1}^k \alpha_i \nabla f_i(x) = \sum_{i = 1}^k \alpha_i \sum_{j = 1}^d c_{ij} \nabla b_j(x) = \sum_{i = 1}^k \sum_{j = 1}^d \alpha_i c_{ij} \nabla b_j(x) = L(x,\alpha) c
	\end{equation*}
	with 
	\begin{equation} \label{eq:c_order}
		c = (c_{11},...,c_{1d},c_{21},...,c_{2d},...,c_{k1},...,c_{kd})^\top \in \R^{k \cdot d}
	\end{equation}
	and
	\begin{align*}
		&L(x,\alpha) := 	\\
		&(\alpha_1 \nabla b_1(x),...,\alpha_1 \nabla b_d(x),\alpha_2 \nabla b_1(x),...,\alpha_2 \nabla b_d(x),...,\alpha_k \nabla b_1(x),...,\alpha_k \nabla b_d(x)) \in \R^{n \times (k \cdot d)}.
	\end{align*}
	Let 
	\begin{equation*}
		\L := \begin{pmatrix}
			L(\bar{x}^1,\bar{\alpha}^1) \\
			\vdots \\
			L(\bar{x}^N,\bar{\alpha}^N)
		\end{pmatrix}
		\in \R^{ (n \cdot N) \times (k \cdot d) }.
	\end{equation*}
	Then \eqref{eq:KKT_exact} is equivalent to the homogeneous linear system
	\begin{equation} \label{eq:KKT_exact_lin}
		\L c = 0.
	\end{equation}
	Since the derivatives of the basis functions are linearly independent, a (nontrivial) function satisfying \eqref{eq:KKT_exact} exists if and only if 
	\begin{equation} \label{eq:rank_L}
		rk(\L) < k \cdot d.
	\end{equation}
	We will now consider two cases for the dimension of system \eqref{eq:KKT_exact_lin}: \\
	\textit{Case 1}: $n \cdot N < k \cdot d$, i.e., \eqref{eq:KKT_exact_lin} is an underdetermined system. In this case, \eqref{eq:rank_L} automatically holds such that \eqref{eq:KKT_exact_lin} possesses at least one nontrivial solution. In other words, $dim(ker(\L)) > 0$. Note that in this case, our approach resembles an interpolation method. In fact, for $n = 1$, $k = 1$ and monomial basis functions, $\L$ is similar to the \textit{Vandermonde matrix} from polynomial interpolation (without the constant column). \\
	\textit{Case 2}: $n \cdot N \geq k \cdot d$, i.e., \eqref{eq:KKT_exact_lin} is a square or overdetermined system. This means that generically, \eqref{eq:KKT_exact_lin} does not have a solution, and we have to check the condition \eqref{eq:rank_L}. In practice, we can use singular value decomposition (SVD) to do this, as the rank of $\L$ equals the number of singular values of $\L$ that are non-zero. In particular, as $rk(\L) = k \cdot d - dim(ker(\L))$, it yields the dimension of the solution space of \eqref{eq:KKT_exact_lin}.
		
	For ease of notation, we make the following definition:
	\begin{definition}
		Let
		\begin{equation*}
			\F : \R^{k \cdot d} \rightarrow C^1(\R^n, \R^k), \quad c \mapsto (f_i)_{i \in \{ 1,...,k \}} = \left( \sum_{j = 1}^d c_{ij} b_j \right)_{i \in \{ 1,...,k \}}
		\end{equation*}
		be the map that maps a coefficient vector $c$ onto the corresponding objective vector $(f_i)_i$ (cf. \eqref{eq:f_in_basis} and \eqref{eq:c_order}).
	\end{definition}	
	
	It is easy to see that $\F$ is linear and by the linear independence of the derivatives of the basis functions, $\F$ is also injective.
	
	\subsection{Finding the best approximation}
	In most applications, one can expect that \eqref{eq:KKT_exact_lin} is overdetermined and that it cannot be solved exactly. Even if there was a solution, we would require exact data to find it, which is numerically impossible. Furthermore, the case where the data is slightly noisy is much more realistic. Therefore, it makes more sense to look for the MOP whose extended Pareto critical set is the best approximation for a given data set, i.e., where $\| Df(\bar{x})^\top \bar{\alpha} \|_2$ is as small as possible for all $(\bar{x},\bar{\alpha}) \in \D$. To this end, consider the problem
	\begin{equation} \label{eq:min_Lc}
		\min_{\| c \|_2 = 1} \| \L c \|_2,
	\end{equation}
	where the vector of coefficients is constrained to the unit sphere $\mathcal{S}^{(k \cdot d) - 1}$ in $\R^{k \cdot d}$ to avoid the trivial solution $c^* = 0$. If $c^*$ is a solution of \eqref{eq:min_Lc} and $f = \F(c^*)$ is the corresponding objective vector, then
	\begin{equation} \label{eq:upper_bound_KKT}
		\| Df(\bar{x})^\top \bar{\alpha} \|_2 = \| L(\bar{x},\bar{\alpha}) c^* \|_2 \leq \| \L c^* \|_2 \quad \forall (\bar{x},\bar{\alpha}) \in \D,
	\end{equation}
	i.e., the optimal value of \eqref{eq:min_Lc} is an upper bound for all $\| Df(\bar{x})^\top \bar{\alpha} \|_2$ with $(\bar{x},\bar{\alpha}) \in \D$. In particular, the optimal value of \eqref{eq:min_Lc} is zero if and only if \eqref{eq:rank_L} holds. Problem \eqref{eq:min_Lc} can easily be solved using SVD (see, e.g., \cite{GGK2014}): Assume that $n \cdot N \geq k \cdot d$, i.e., \eqref{eq:KKT_exact_lin} is overdetermined. Let
	\begin{equation*}
		\L = U S V^\top
	\end{equation*}
	be the SVD of $\L$ with sorted singular values $s_1 \leq s_2 \leq ... \leq s_{k \cdot d}$. Let $v_1, ..., v_{k \cdot d} \in \R^{k \cdot d}$ be the right-singular vectors of $\L$, i.e., the columns of $V$. Then
	\begin{equation} \label{eq:sol_min_Lc}
		\min_{\| c \|_2 = 1} \| \L c \|_2 = s_1 \quad \text{and} \quad \argmin_{\| c \|_2 = 1} \| \L c \|_2 = \vspan( \{ v_i : s_i = s_1 \} ) \cap \mathcal{S}^{(k \cdot d) - 1}.
	\end{equation}
	Consequently, $s_1$ is a measure for how well the data set $\D$ can be approximated with the extended Pareto critical set of an MOP where the objective functions are linear combinations of the basis functions in $\B$. Furthermore, the singular values of $\L$ can be used to determine the dimension of the space of approximating objective vectors.
	
	\begin{algorithm} 
		\caption{Generate objective vector from data}
		\label{algo:MOP_from_data}
		Given: Data set $\D \subseteq \R^n \times \Delta_k$, basis functions $\B \subseteq C^1(\R^n,\R)$, threshold $\overline{s}$.
		\begin{algorithmic}[1] 
			\State Assemble $\L$.
			\State Calculate the SVD of $\L$ with singular values $s_1 \leq s_2 \leq ... \leq s_{k \cdot d}$ and right-singular vectors $v_1,...,v_{k \cdot d}$.
			\State Identify the indices $I = \{1, ..., i^* \}$, $i^* \leq k \cdot d$, such that $s_{i} \leq \overline{s}$ for all $i \in I$.
			\State Choose an element 
				\begin{equation*}
					c^* \in \vspan( \{ v_i : i \in I \} ) \setminus \{ 0 \} \subseteq \R^{k \cdot d}
				\end{equation*}							
			with $\|c^*\|_2 = 1$.
			\State Assemble the objective vector $f = \F(c^*)$ as in \eqref{eq:f_in_basis}.
		\end{algorithmic} 
	\end{algorithm}

	Algorithm \ref{algo:MOP_from_data} summarizes the numerical procedure which follows from the above considerations.
	The resulting approximation then satisfies the following property:	
	\begin{theorem} \label{theorem:convergence}
		Let $f$ be the result of Algorithm \ref{algo:MOP_from_data} and $s_{i^*}$ be the largest singular value less or equal to $\bar{s}$. Then
		\begin{equation*}
			\| Df(\bar{x})^\top \bar{\alpha} \|_2 \leq s_{i^*} \quad \forall (\bar{x},\bar{\alpha}) \in \D.
		\end{equation*}
		In particular, if $s_{i^*} = 0$, then all $(\bar{x},\bar{\alpha}) \in \D$ are extended Pareto critical for the MOP with objective vector $f$.
	\end{theorem}
	\begin{proof}
		Let $c^*$ be the coefficient vector in step $4$ such that $f = \F(c^*)$. Then there is some $\lambda \in \R^{k \cdot d}$ with $c^* = V \lambda$, $\lambda_{i^* + 1} = ... = \lambda_{k \cdot d} = 0$ and $1 = \| c^* \|_2 = \| V \lambda \|_2 = \| \lambda \|_2$. Thus
		\begin{equation*}
			\| \L c^* \|_2 = \| \L V \lambda \|_2 = \| U S \lambda \|_2 = \| S \lambda \|_2 = \sqrt{ \sum_{i = 1}^{k \cdot d} s_i^2 \lambda_i^2} \leq s_{i^*} \sqrt{ \sum_{i = 1}^{k \cdot d} \lambda_i^2} = s_{i^*} \| \lambda \|_2 = s_{i^*}.
		\end{equation*}
		Combining this with \eqref{eq:upper_bound_KKT} completes the proof.
	\end{proof}
	
	Some properties of Algorithm \ref{algo:MOP_from_data} are highlighted in the following remark.
	\begin{remark} \label{rem:algorithm}
		\begin{enumerate}[a)]
			\item Algorithm \ref{algo:MOP_from_data} can also be applied when \eqref{eq:KKT_exact_lin} is underdetermined, i.e., when $n \cdot N \leq k \cdot d$, by treating $v_{(n \cdot N) + 1}, ..., v_{k \cdot d}$ as right-singular vectors to the ``singular value'' zero.
			\item In general, if $i^* > 1$, there is no obvious choice for $c$ in step 4. A possible approach is to choose $c$ as sparse as possible (using, e.g., $L_1$ minimization \cite{T1996}). This can be very advantageous for interpretability, see also \cite{BPK2016} for sparse identification in the dynamical systems context.
			\item It is important to note that by construction, if $s_{i^*} = 0$, we can only guarantee that $\D$ is a subset of the extended Pareto critical set of $f$. It is possible for the extended Pareto critical set to contain more than just $\D$ (cf. Example \ref{example:unit_circle}). Therefore, there are cases where the smallest singular value is $0$, but the corresponding MOP might not be desirable.
			\item According to \eqref{eq:sol_min_Lc}, if $s_{i^*} = 0$, we can take any element of $\vspan( \{ v_i : i \in I \} ) \setminus \{ 0 \}$ in step $3$ and do not need to normalize it.
		\end{enumerate}
	\end{remark}	
	
	We will conclude this section with a brief discussion on the choice of the set of basis functions $\B$. It should satisfy the following requirements:	
	\begin{enumerate}[(i)]
		\item The derivatives of the basis functions should be linearly independent to avoid trivial solutions. (In particular, this implies that the representation of the derivatives of elements of $\vspan(\B)$ via coefficients of the derivatives of elements of $\B$ is unique.)
		\item Since we have to evaluate the derivatives of the basis functions in every data point in $\D_x$ for the assembly of $\L$, the evaluation of these derivatives should be efficient.
		\item In practice, an initial, a priori choice of $\B$ will often be insufficient. Thus, it should be possible to increase the quality of the approximation by increasing the size of $\B$ without much effort.
	\end{enumerate}	
	
	An intuitive choice for $\B$ are the monomials in $n$ variables up to degree $l \in \N$, i.e.,
	\begin{equation*}
		\B = \{ x_1^{l_1} x_2^{l_2} \cdots x_n^{l_n} : l_i \in \N \cup \{ 0 \}, i \in \{1,...,n\}, 0 < l_1 + l_2 + ... + l_n \leq l \}.
	\end{equation*}
	It is easy to see that (i) and (ii) are satisfied for this choice. For (iii), the Stone-Weierstrass theorem (cf.~\cite{R1991}) implies that for any compact $D \subseteq \R^n$ and any $g \in C^1(D,\R)$, there exists a sequence of polynomials on $D$ that converges to $g$ (with respect to $\| \cdot \|_\infty$). Thus, for ${ g \in C^1(\R^n,\R^k) }$, uniform convergence with polynomial functions can be guaranteed component-wise on compact subsets of $\R^n$. Therefore, for the rest of this article, we will always consider the monomials up to a fixed degree as the set of basis functions.
	
	\section{Application 1: Constructing objectives from clean and noisy data} \label{sec:constructingObjectives}
	In this section, we will show how the results from Section \ref{sec:generating_mops_from_data} can be utilized to construct objective functions from clean and noisy data. Our first example will be the construction of test problems for MOP solvers, where the data comes from a discretization of some continuous (i.e., non-discrete) set. In the second example, we will consider a stochastic MOP, where we will reconstruct the expected value of the objective vector using stochastic (i.e., noisy) solution data.
	
	\subsection{Inferring objectives from exact data} \label{subsec:exactData}
	Test problems and generators of test problems are an important tool to investigate the behavior and to benchmark MOP solvers (cf. \cite{D1999, ZZZ2008, KWP2016}). The idea is to interpret our method from Section \ref{sec:generating_mops_from_data} as a way to generate MOPs where we already know the extended Pareto critical set. Instead of finitely many extended Pareto critical points, we here want to prescribe the complete set. Unfortunately, as discussed in Section \ref{sec:generating_mops_from_data}, this will generically fail, as we can only assure that the prescribed data set is Pareto critical for the MOP resulting from Algorithm \ref{algo:MOP_from_data} if $n \cdot N < k \cdot d$, i.e., 
	\begin{equation*}
		N < \frac{k \cdot d}{n}.
	\end{equation*}
	Nonetheless, it turns out that if we use monomials as basis functions, the resulting space of polynomials is large enough to contain objective vectors for many non-trivial classes of infinite data sets.
	
	When prescribing an infinite data set $\D^\infty = (\D^\infty_x, \D^\infty_\alpha) \subseteq \R^n \times \Delta_k$, we have to ensure that $\D^\infty$ has the properties of an extended Pareto critical set from a theoretical point of view. This is obviously the difficult part of this approach and requires some knowledge about the structure of (extended) Pareto critical sets. In the following, we will briefly summarize the generic implications of the results from Section \ref{sec:introduction_MOP}:
	\begin{itemize}
		\item According to \cite{H2001}, $\D^\infty$ should (locally) be a differentiable manifold. In practice, this means that similar $\bar{x} \in \D^\infty_x$ should have similar $\bar{\alpha} \in \D^\infty_\alpha$.
		\item Following \cite{GPD2019}, for points $(\bar{x},\bar{\alpha}) \in \D^\infty$ where $\bar{x}$ lies on the edge of $\D^\infty_x$, we have to ensure that $\bar{\alpha}_j = 0$ for some $j \in \{1,...,k\}$. In particular, for multiple Pareto critical points on the same edge, the same component of the corresponding $\bar{\alpha}$ has to be zero.
	\end{itemize}
	
	After constructing a data set $\D^\infty$ with the properties mentioned above, we consider a pointwise discretization $\D \subseteq \D^\infty$ for large $N = |\D|$ and apply Algorithm \ref{algo:MOP_from_data}. If the smallest singular value of $\L$ is zero, $\D^\infty$ will be in the extended Pareto critical set of the resulting MOP.

	
	\begin{example} \label{example:unit_circle}
		In this example, we will generate an MOP with two objective functions where the Pareto critical set is the unit circle $\mathcal{S}^1$ in $\R^2$. To this end, let $N \in \N$ and
		\begin{equation*}
			\D := \{ (\bar{x}^j, \bar{\alpha}^j) \in \R^2 \times \Delta_2 : j \in \{1,...,N\} \},
		\end{equation*}
		where
		\begin{equation} \label{eq:example_unit_circle_data}
			\bar{x}^j := \begin{pmatrix}
				cos(2 \pi \frac{j}{N}) \\
				sin(2 \pi \frac{j}{N})
			\end{pmatrix}
			\quad \text{and} \quad
			\bar{\alpha}^j := \begin{pmatrix}
				0.5 (cos(4 \pi \frac{j}{N}) + 1) \\
				1 - 0.5 (cos(4 \pi \frac{j}{N}) + 1)
			\end{pmatrix}.
		\end{equation}
		The $\bar{x}^j$ are points distributed equidistantly on $\mathcal{S}^1$ and the KKT vectors $\bar{\alpha}^j$ are $\frac{N}{2}$-periodic between $0$ and $1$. While the choice of the $\bar{x}^j$ is straight-forward, the selection of the corresponding KKT vectors is less obvious. We chose them periodically to ensure that there is no jump from $\bar{x}^N$ to $\bar{x}^1$. (Additionally, we chose a different ``frequency'' than for the $\bar{x}^j$, i.e., $4 \pi$ instead of $2 \pi$, to avoid unwanted structures in the data.) We choose monomials up to degree $3$ as basis functions, i.e.,
		\begin{equation*}
		\B := \{ x_1, x_1^2, x_1^3, x_2, x_1 x_2, x_1^2 x_2, x_2^2, x_1 x_2^2, x_2^3 \},
		\end{equation*}
		and use $N = 1000$ data points. Figure \ref{fig:example_circle_results}(a) shows the sorted singular values of the resulting $\L \in \R^{2000 \times 18}$.
		The first two singular values $s_1 = 3.92 \cdot 10^{-15}$ and $s_2 = 9.69 \cdot 10^{-15}$ are small and there is an obvious gap from $s_2$ to $s_3 = 5.41$. Since $s_1$ and $s_2$ are both close to zero, we choose the threshold $\bar{s} = s_2$, i.e., $I = \{1, 2\}$, in step 3 of Algorithm \ref{algo:MOP_from_data}. The corresponding columns of $V$ are given by
		\begin{align*}
			v_1 &= (-0.9040, 0, 0.3013, 0, 0, 0, 0, 0, 0.010, 0, 0, 0.3013, -0.030, 0, 0, 0, 0, 0.010)^\top, \\
			v_2 &= (-0.030, 0, 0.010, 0, 0, 0, 0, 0, -0.3013, 0, 0, 0.010, 0.9040, 0, 0, 0, 0, -0.3013)^\top.
		\end{align*}
		In this example, it is easy to see that there is a certain pattern in $v_1$ and $v_2$, so that we can write
		\begin{equation} \label{eq:example_circle_theoretical_span}
			\vspan(\{v_1, v_2\}) = \{ ( -3 \sigma_1, 0, \sigma_1, 0, 0, 0, 0, 0, \sigma_2, 0, 0, \sigma_1, -3 \sigma_2, 0, 0, 0, 0, \sigma_2 )^\top : \sigma_1, \sigma_2 \in \R \}.
		\end{equation}
		Unfortunately, not all elements of $\vspan(\{v_1, v_2\}) \setminus \{ 0 \}$ in step 4 lead to desirable objective vectors. To see this, consider the element $c$ corresponding to $\sigma_1 = 0$ and $\sigma_2 = 1$, i.e.,
		\begin{equation*}
			c = (0,0,0,0,0,0,0,0,1,0,0,0,-3,0,0,0,0,1)^\top.
		\end{equation*}
		The corresponding objective vector is given by
		\begin{equation} \label{eq:example_circle_degenerated_objective}
			\F(c)(x) = \begin{pmatrix}
				x_2^3 \\
				x_2^3 - 3 x_2
			\end{pmatrix}.
		\end{equation}
		For this objective vector, the extended Pareto critical set indeed contains the given data set. However, the entire Pareto critical set for this problem is given by $\R \times [-1,1]$, hence it contains significantly more than what we prescribed. In this case, the degeneracy is caused by the fact that this objective vector does not depend on $x_1$. A better choice for $c$ would be, e.g., $\sigma_1 = 1$ and $\sigma_2 = 1$, resulting in
		\begin{equation*}
			c = (-3,0,1,0,0,0,0,0,1,0,0,1,-3,0,0,0,0,1)^\top.
		\end{equation*}
		The corresponding objective vector $f = \F(c)$ is given by
		\begin{equation*}
			f(x) = \begin{pmatrix}
				- 3 x_1 + x_1^3 + x_2^3 \\
				- 3 x_2 + x_1^3 + x_2^3
			\end{pmatrix}.
		\end{equation*}
		One can show that for this objective vector, the KKT conditions are indeed equivalent to
		\begin{align*}
			&x_1^2 + x_2^2 = 1, \\
			&\alpha_1 = x_1^2, \\
			&\alpha_2 = x_2^2,
		\end{align*}
		i.e., the Pareto critical set is precisely $\mathcal{S}^1$ with the corresponding KKT vectors given in \eqref{eq:example_unit_circle_data}. (In particular, we did not need to normalize $c$ in step 4 of Algorithm \ref{algo:MOP_from_data} in this case.) Figures \ref{fig:example_circle_results}(b) and (c) show the Pareto critical set and the image of the Pareto critical set under $f$.
		\begin{figure}[t] 
			\parbox[b]{0.32\textwidth}{
				\centering 
				\includegraphics[width=.32\textwidth]{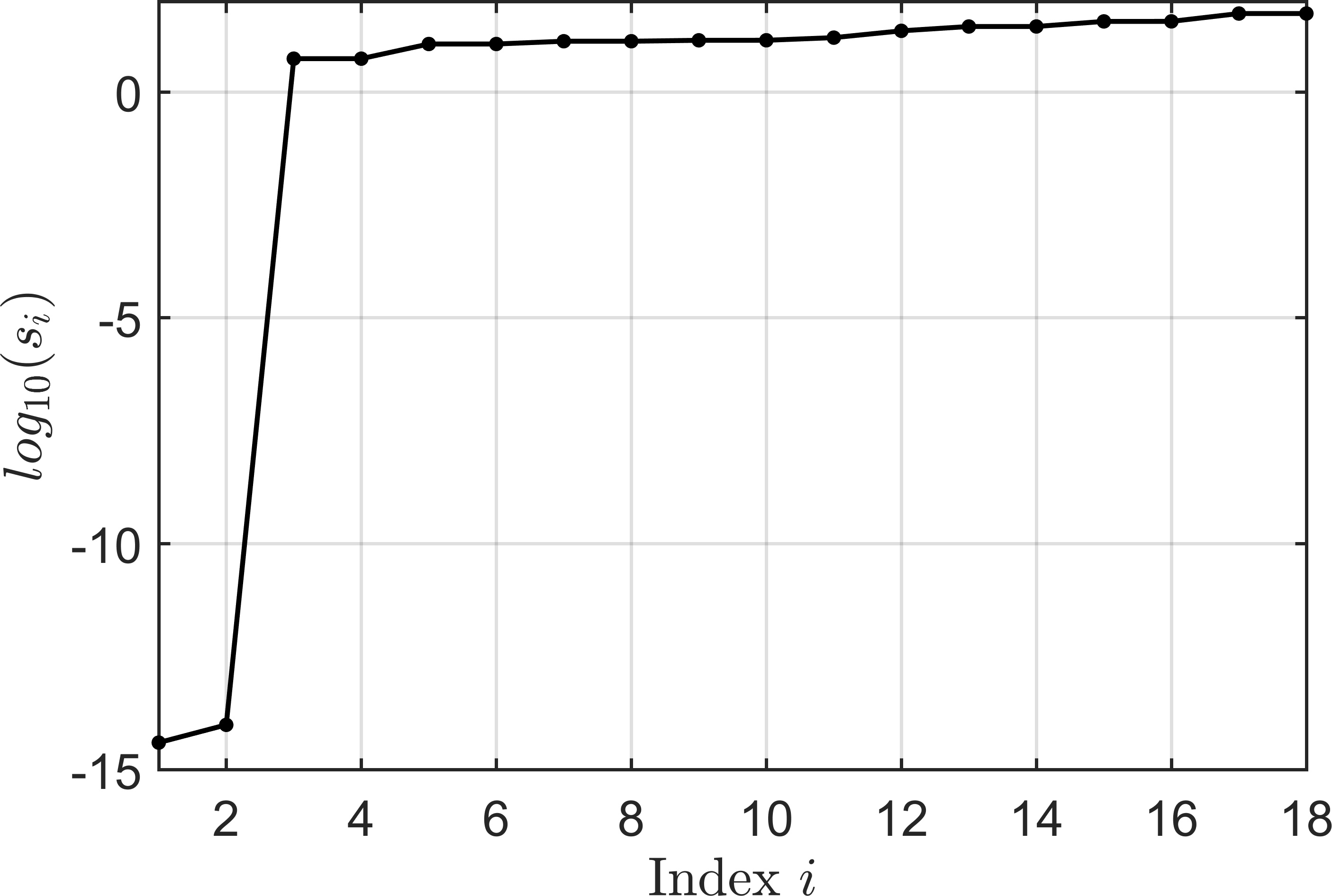}\\[1em]
				\textbf{(a)}
			}
			\parbox[b]{0.32\textwidth}{
				\centering 
				\includegraphics[width=0.32\textwidth]{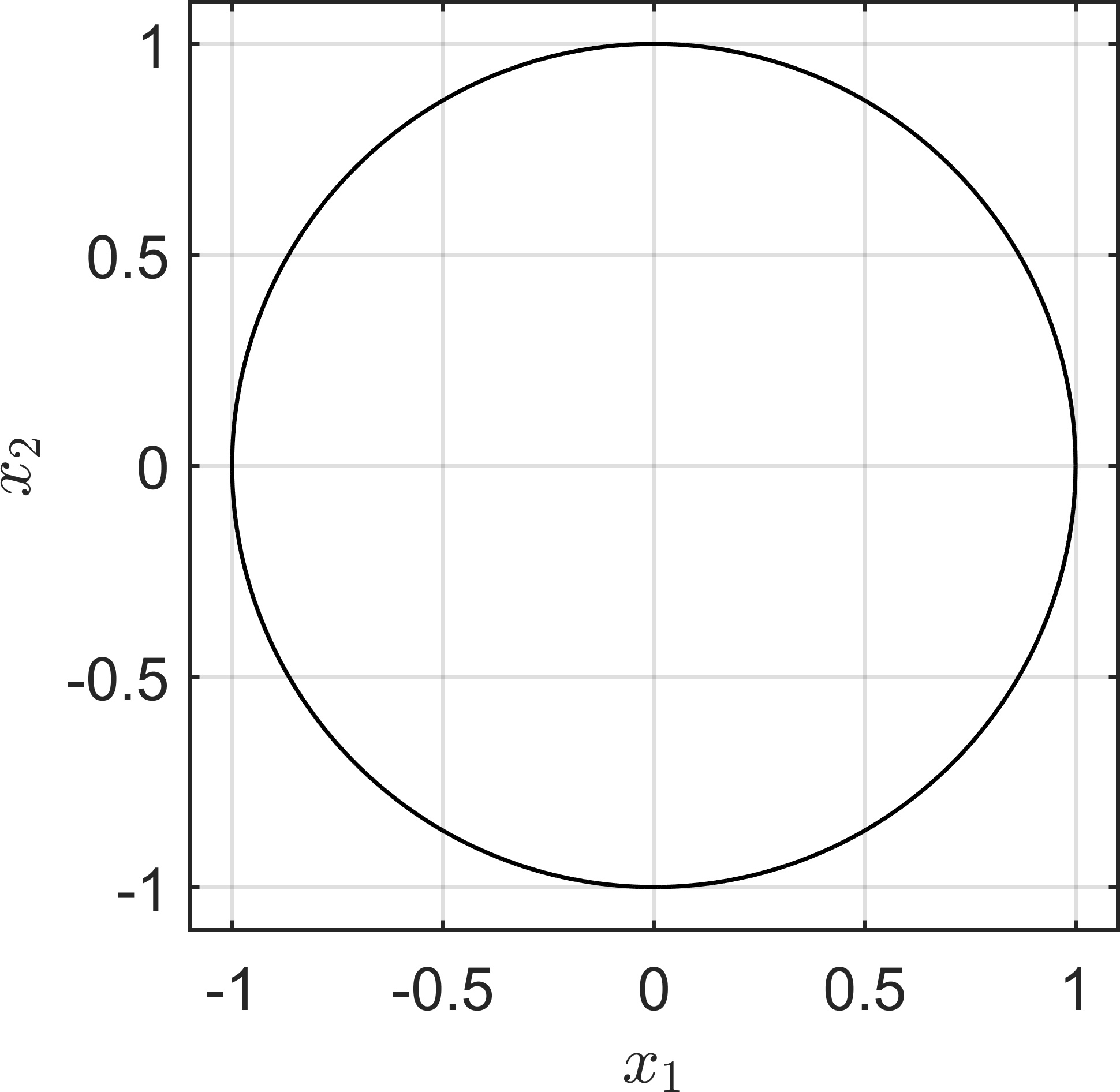}\\[1em]
				\textbf{(b)}
			}
			\parbox[b]{0.32\textwidth}{
				\centering 
				\includegraphics[width=0.32\textwidth]{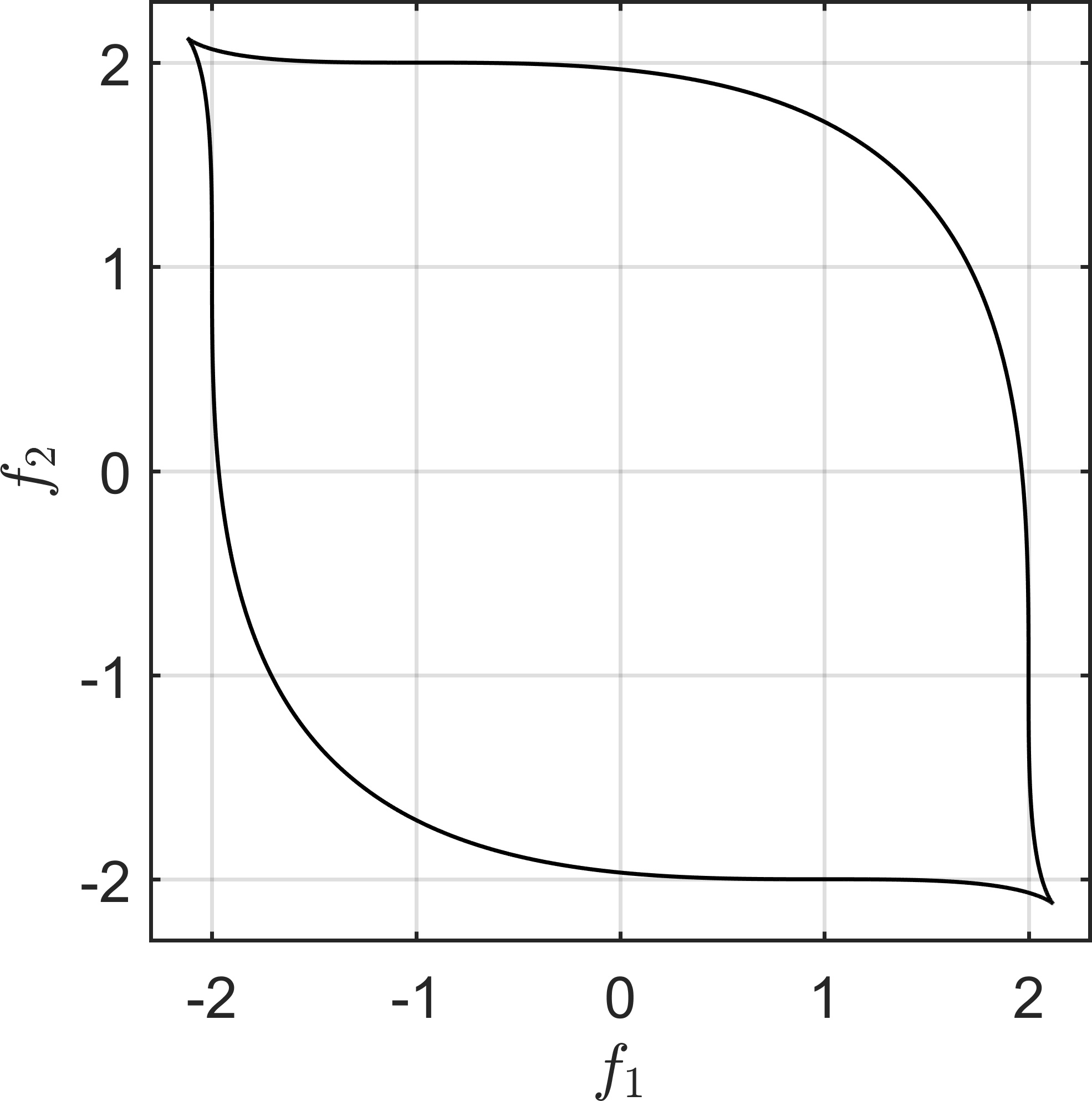}\\[1em]
				\textbf{(c)}
			}
			\caption{\textbf{(a)} Singular values of $\L$ in Example \ref{example:unit_circle}. \textbf{(b)} Pareto critical set of $f$. \textbf{(c)} Image of the Pareto critical set of $f$.}
			\label{fig:example_circle_results}
		\end{figure}	
	\end{example}
	
	Example \ref{example:unit_circle} shows how the results from Section \ref{sec:generating_mops_from_data} can be used to derive an explicit expression for an objective vector for a prescribed data set. The following example shows that we can even derive more general formulas.
	
	\begin{example} \label{example:ellipses}
		Using the same strategy as in Example \ref{example:unit_circle}, it is possible to numerically verify that arbitrary ellipses can be represented as Pareto critical sets of polynomial MOPs. For $a,b \in \R^{>0}$, we merely have to replace the $\bar{x}^j$ from Example \ref{example:unit_circle} by
		\begin{equation*}
			\bar{x}^j := \begin{pmatrix}
				a \cdot cos(2 \pi \frac{j}{N}) \\
				b \cdot sin(2 \pi \frac{j}{N})
			\end{pmatrix}.
		\end{equation*}
		In this case, if we consider the analogous expression to \eqref{eq:example_circle_theoretical_span}, we see that variations of $a$ and $b$ only influence a single component, respectively. In general, the following pattern can be recognized:
		\begin{equation*}
			\vspan(\{v_1, v_2\}) = \{ ( -3 a^2 \sigma_1, 0, \sigma_1, 0, 0, 0, 0, 0, \sigma_2, 0, 0, \sigma_1, -3 b^2 \sigma_2, 0, 0, 0, 0, \sigma_2 )^\top : \sigma_1, \sigma_2 \in \R \},
		\end{equation*}
		which leads to the following conjecture: Let
		\begin{equation*}
			f : \R^2 \rightarrow \R^2, \quad x \mapsto
			\begin{pmatrix}
				-3 a^2 x_1 + x_1^3 + x_2^3 \\
				-3 b^2 x_2 + x_1^3 + x_2^3
			\end{pmatrix}.
		\end{equation*}
		Then
		\begin{equation*}
			P_c = \left\{ x \in \R^2 : \frac{x_1^2}{a^2} + \frac{x_2^2}{b^2} = 1 \right\}
		\end{equation*}
		and the KKT vector corresponding to $x \in P_c$ is given by $\alpha = \left( \frac{x_1^2}{a^2}, \frac{x_2^2}{b^2} \right)^\top$. After deriving this conjecture numerically, it is straight-forward to prove that it actually holds.
	\end{example}
	
	In Examples \ref{example:unit_circle} and \ref{example:ellipses}, the symbolic expressions could easily be verified. In particular, in step 4 of Algorithm \ref{algo:MOP_from_data} we were able to choose $c$ such that the Pareto critical set did not contain more than what we intended, i.e., $P_c$ was precisely the unit circle or an ellipse. This obviously only works if the data set is sufficiently well-structured. The following example shows a more complicated case.
	
	\begin{example} \label{example:3_con_comp}
		We are now searching for an MOP where the Pareto critical set contains three connected components $C_i$, $i = 1,2,3$, given by the following three non-intersecting straight lines:
		\begin{equation*}
			C_i = p_i + [0,1] \cdot \frac{1}{4} \frac{q_i}{\|q_i\|_2} \subseteq \R^2
		\end{equation*}
		with
		\begin{align*}
			p_1 = \begin{pmatrix}
				0.15 \\
				-0.20
			\end{pmatrix},~
			q_1 = \begin{pmatrix}
				0.47 \\
				0.04
			\end{pmatrix},~
			p_2 = \begin{pmatrix}
				0.47 \\
				-0.32
			\end{pmatrix},~
			q_2 = \begin{pmatrix}
				0.40 \\
				0.14
			\end{pmatrix},~
			p_3 = \begin{pmatrix}
				0.37 \\
				0.18
			\end{pmatrix},~
			q_3 = \begin{pmatrix}
				0.38 \\
				0.28
			\end{pmatrix}.
		\end{align*}
		They are shown in Figure \ref{fig:example_3_con_comp_solution}(a). For $\D_x$ we choose $N_c = 500$ equidistant points on each $C_i$, the corresponding $\D_\alpha$ are chosen linearly from $(0,1)^\top$ to $(1,0)^\top$, and we again use monomials as basis functions. When dealing with more complex data sets, we first have to estimate the required degree of monomials for a satisfactory approximation. To this end, we repeat step 2 of Algorithm \ref{algo:MOP_from_data} for different maximal degrees. The smallest singular value depending on the maximal degree of the monomials is shown in Figure \ref{fig:example_3_con_comp_singular_values}(a).
		\begin{figure}[ht] 
			\parbox[b]{0.49\textwidth}{
				\centering 
				\includegraphics[width=0.4\textwidth]{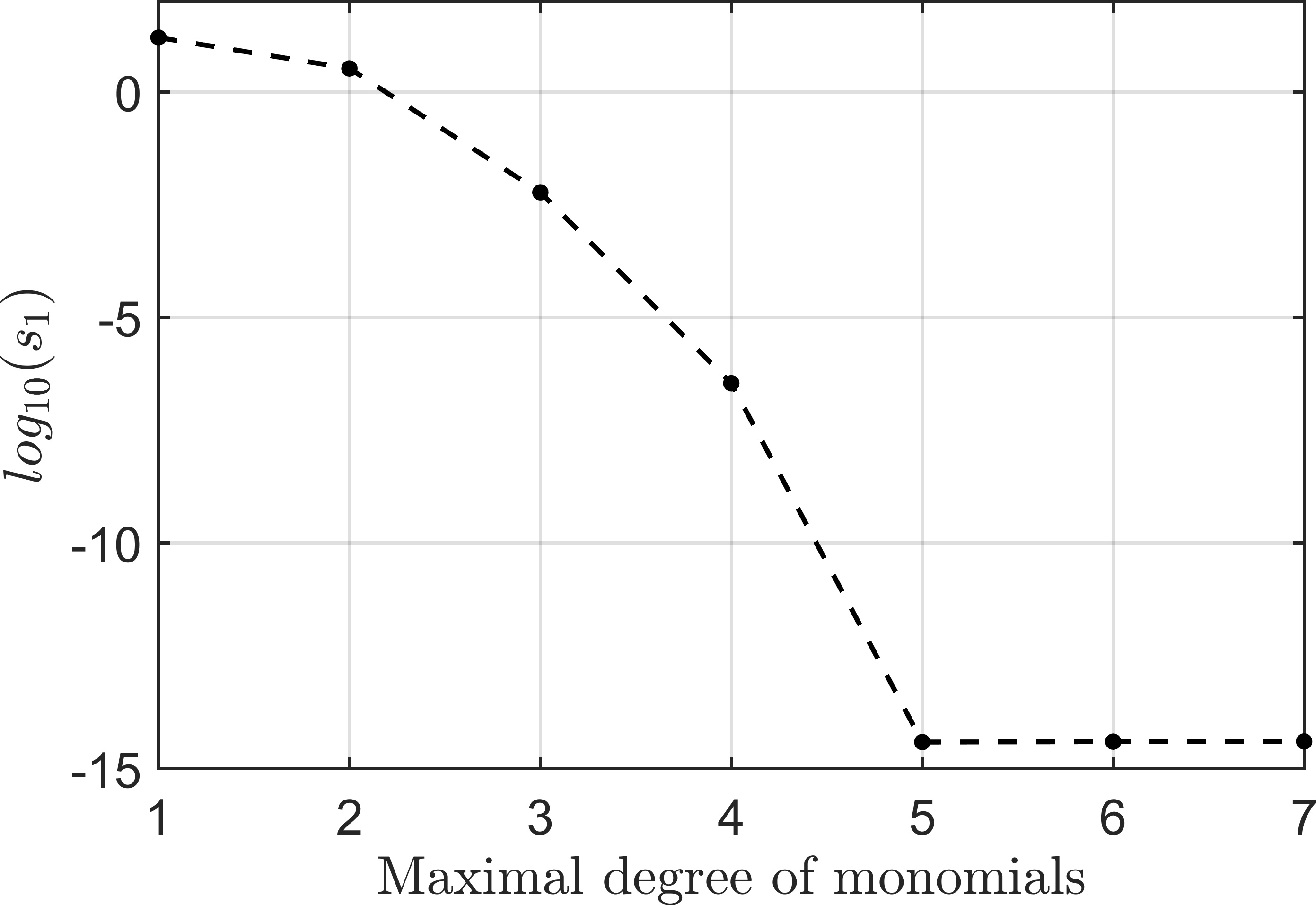}\\
				\textbf{(a)}
			}
			\parbox[b]{0.49\textwidth}{
				\centering 
				\includegraphics[width=0.4\textwidth]{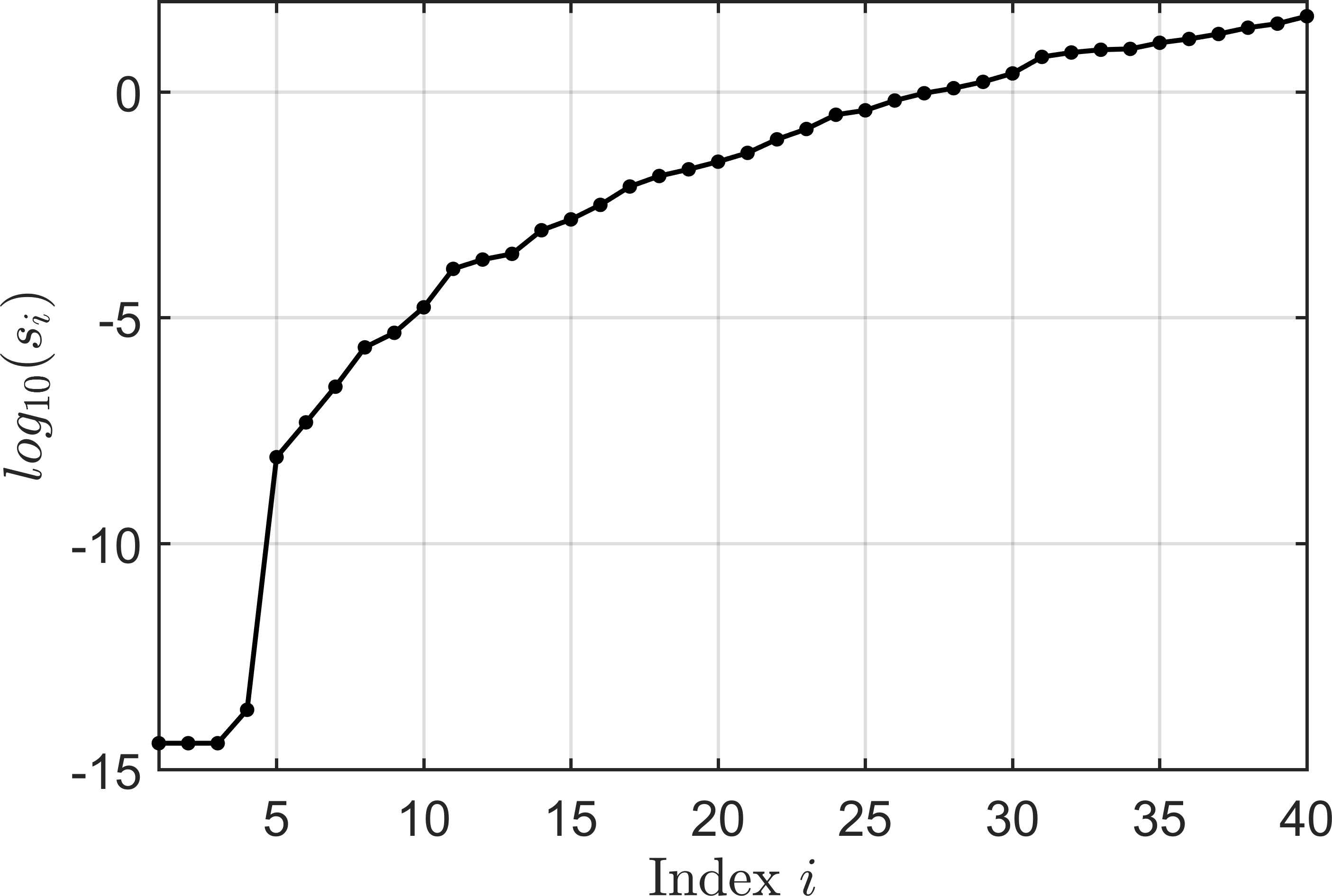}\\
				\textbf{(b)}
			}
			\caption{\textbf{(a)} Smallest singular value of $\L$ for different degrees of the monomial basis in Example \ref{example:3_con_comp}. \textbf{(b)} Singular values of $\L$ for degree $5$.}
			\label{fig:example_3_con_comp_singular_values}
		\end{figure}
		We see that the monomials up to a degree of $5$ are a promising choice, since the smallest singular values do not decrease further after that. Figure \ref{fig:example_3_con_comp_singular_values}(b) shows all singular values for this set of basis functions. There is a relatively large gap from $s_4 = 2.09 \cdot 10^{-14}$ to $s_5 = 8.17 \cdot 10^{-9}$, suggesting that $\bar{s} = s_4$, i.e., $I = \{1,2,3,4\}$, in step 3 of Algorithm \ref{algo:MOP_from_data} is a good choice. In this case, there is no obvious way to obtain an expression like \eqref{eq:example_circle_theoretical_span} for $\vspan(\{ v_1,v_2,v_3,v_4 \})$, which is why we choose $c = \frac{v_1}{\| v_1 \|_2}$ in step 4. The Pareto critical set of $f = \F(c)$ and its image are shown in Figure \ref{fig:example_3_con_comp_solution}.
		\begin{figure}[ht] 
			\parbox[b]{0.49\textwidth}{
				\centering 
				\includegraphics[width=0.35\textwidth]{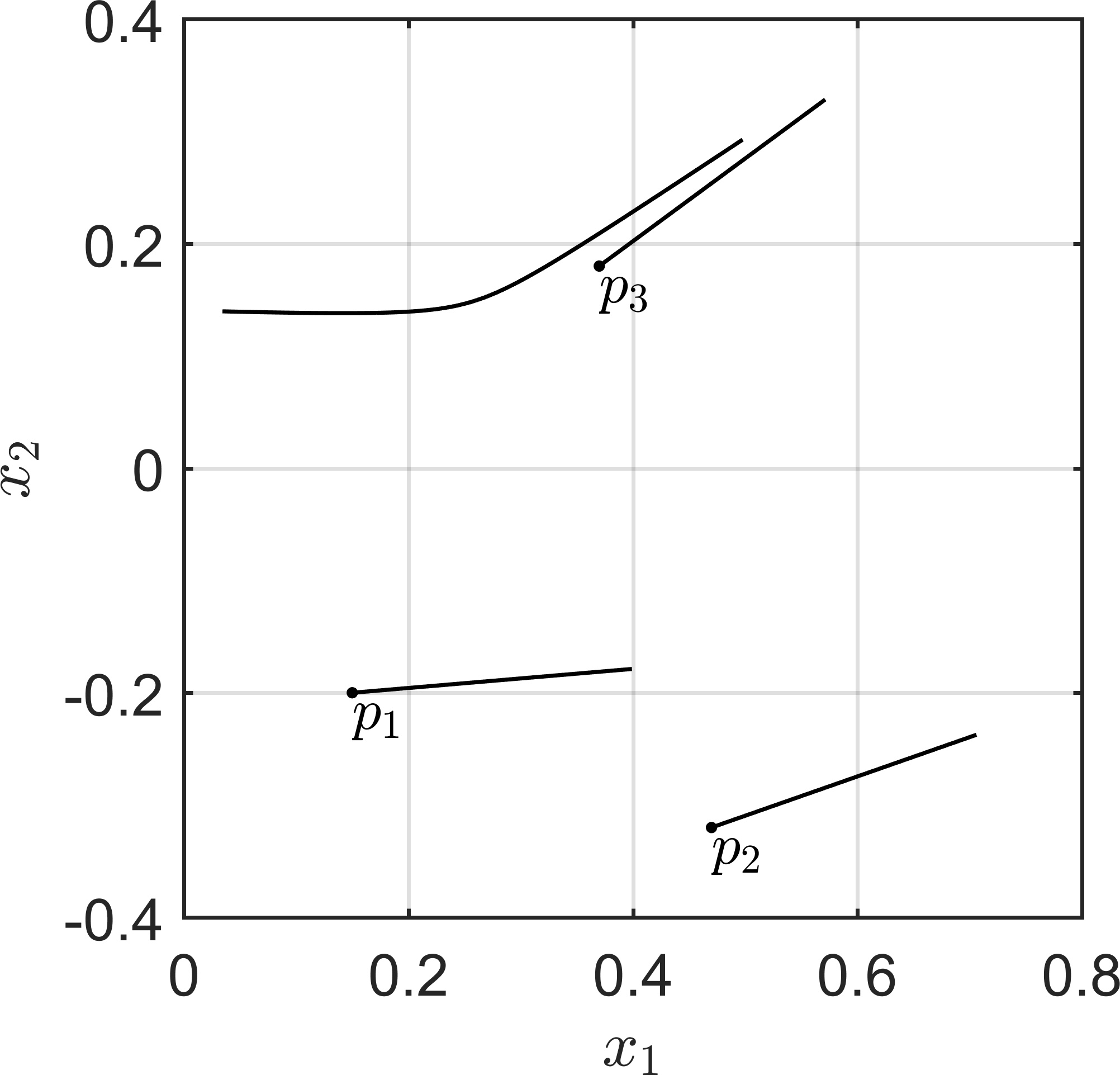}\\
				\textbf{(a)}
			}
			\parbox[b]{0.49\textwidth}{
				\centering 
				\includegraphics[width=0.35\textwidth]{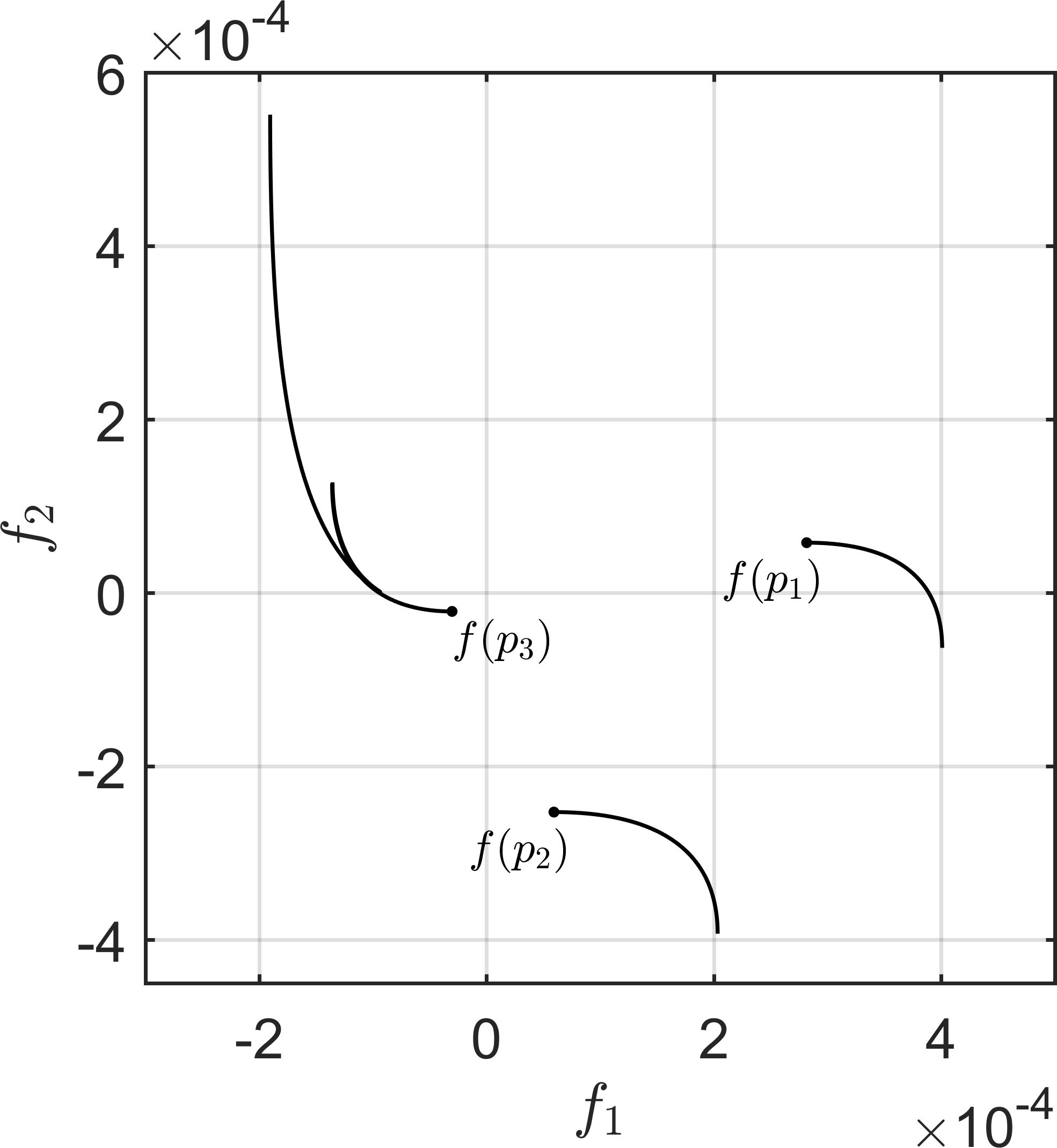}\\
				\textbf{(b)}
			}
			\caption{\textbf{(a)} Pareto critical set of $f$ in Example \ref{example:3_con_comp}. \textbf{(b)} Image of the Pareto critical set of $f$.}
			\label{fig:example_3_con_comp_solution}
		\end{figure}
		As expected from the small singular values, the given data set is approximated almost perfectly. Unfortunately, we observe an additional connected component that is not contained in the data. Since we are unable to influence properties outside the given data set $\D$, additional Pareto critical points can be expected in the general case. However, these can be identified subsequently via comparison to the data set $\D$ in many cases.
	\end{example}
	
	\subsection{Inferring objectives from noisy data}\label{subsec:inexactData}
	In the previous examples, we assumed that we have precise data $\D$ that we want to approximate with an extended Pareto critical set. However, there are many cases where this assumption is unrealistic, for instance real-world applications where the data stems from numerical simulations or measurements. Another example is stochastic multiobjective optimization, which we will consider here. We will only give a brief introduction on this topic and refer to \cite{FX2011} for a more detailed discussion. 
	
	Let $\xi \in \R^m$ be a random vector and $f : \R^n \times \R^m \rightarrow \R^k$. For $x \in \R^n$ let $\mathbb{E}[f(x,\xi)]$ be the (component-wise) expected value of $f(x,\xi)$. For $F(x) := \mathbb{E}[f(x,\xi)]$ we consider the stochastic multiobjective optimization problem
	\begin{equation} \label{eq:SMOP}
		\min_{x \in \R^n} F(x). \tag{SMOP}
	\end{equation}
	Since we cannot evaluate $F$ directly, in practice the sample average
	\begin{equation}
		\tilde{f}^{N_s}(x) = \frac{1}{N_s} \sum_{j = 1}^{N_s} f(x,\xi^j) \approx F(x)
	\end{equation}	 
	is used, where $\xi^1,...,\xi^{N_s}$ are identically independent distributed samples of $\xi$. Using this approximation, we consider the \emph{Sample Average Approximation} problem
	\begin{equation} \label{eq:SAA}
		\min_{x \in \R^n} \tilde{f}^{N_s}(x). \tag{SAA}
	\end{equation}
	For $N_s = \infty$ the solutions of \eqref{eq:SMOP} and \eqref{eq:SAA} coincide. Otherwise, for a finite $N_s \in \N$, we can only expect the solution of \eqref{eq:SAA} to be an approximation of the solution of \eqref{eq:SMOP}. In other words, we can consider the solution of \eqref{eq:SAA} as inexact data of the solution of the original problem \eqref{eq:SMOP} and use our approach to approximate the original solution and objective vector $F$. 
	We illustrate this approach on the following Multiobjective Stochastic Location Problem from \cite{FX2011}.
	\begin{example} \label{example:stochastic}
		Let $a := (-1,-1)^\top$ and $\xi := (\xi_1,0)^\top$ be a random vector, where $\xi_1$ is uniformly distributed on $[0,2]$. Let
		\begin{align*}
			f(x,\xi) := 
			\begin{pmatrix}
				\| x - a \|_2^2 \\
				\| x - \xi \|_2^2
			\end{pmatrix}.
		\end{align*}
		In this case, we have 
		\begin{equation*}
			F(x) = \mathbb{E}[f(x,\xi)] = 
			\begin{pmatrix}
				\| x - a \|_2^2 \\
				\| x - (1,0)^\top \|_2^2 + 1/3				
			\end{pmatrix}
			=
			\begin{pmatrix}
				2 x_1 + x_1^2 + 2 x_2 + x_2^2 + 2 \\
				-2 x_1 + x_1^2 + x_2^2 + 4/3				
			\end{pmatrix}
		\end{equation*}
		so the Pareto critical (and in this case Pareto optimal) set of \eqref{eq:SMOP} is given by the line connecting $a$ and $(1,0)^\top$. The KKT vector corresponding to $x = t a + (1-t) (1,0)^\top$, $t \in [0,1]$, is given by $\alpha = (t,1-t)^\top$. If we solve \eqref{eq:SAA} instead of \eqref{eq:SMOP} (via the smoothing Chebyshev scalarization described in \cite{FX2011}), we obtain an approximation of the original Pareto set as in Figure \ref{fig:example_stochastic_data}(a).
		\begin{figure}[ht] 
			\parbox[b]{0.32\textwidth}{
				\centering 
				\includegraphics[width=0.32\textwidth]{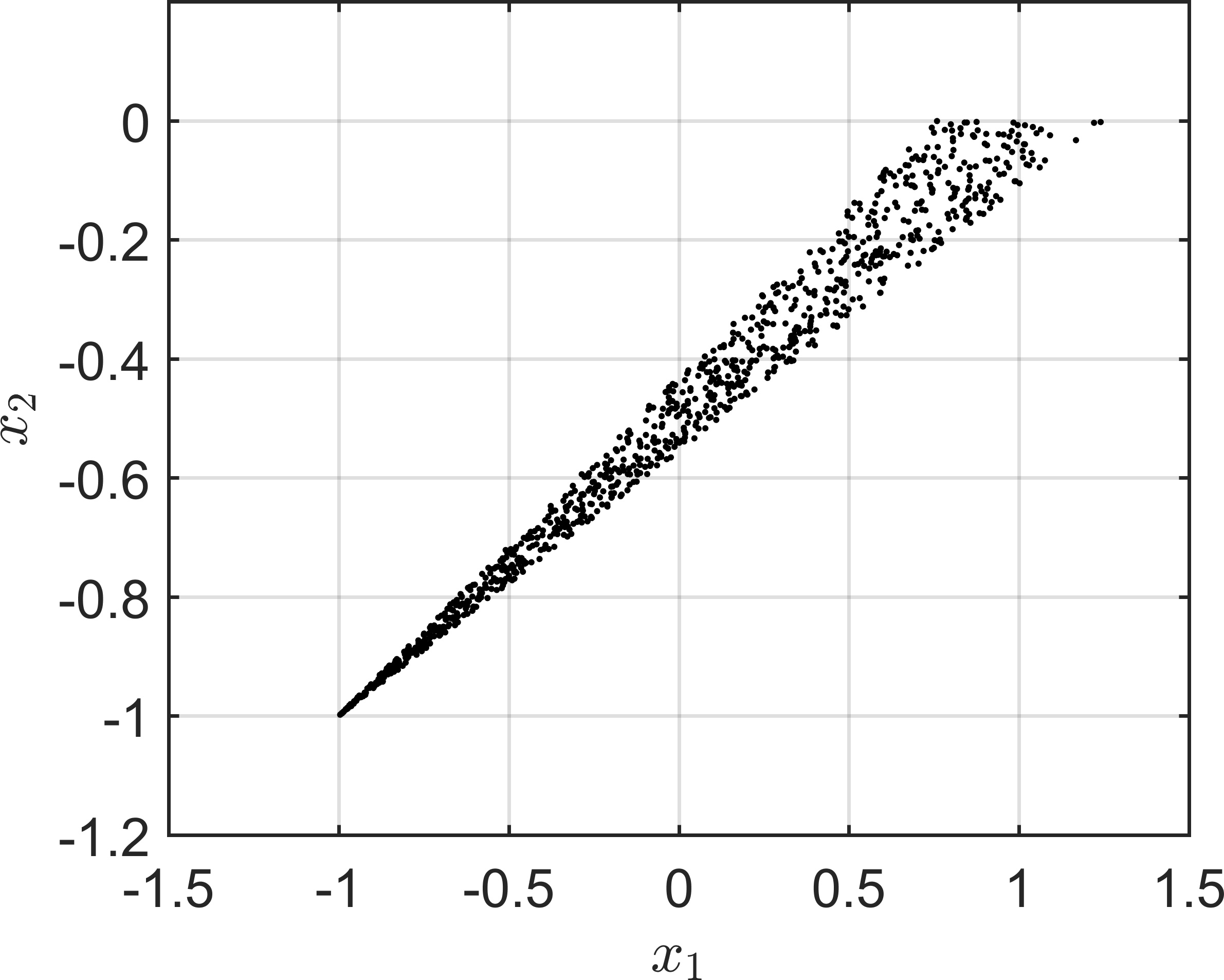}\\[1em]
				\textbf{(a)}
			}
			\parbox[b]{0.32\textwidth}{
				\centering 
				\includegraphics[width=0.32\textwidth]{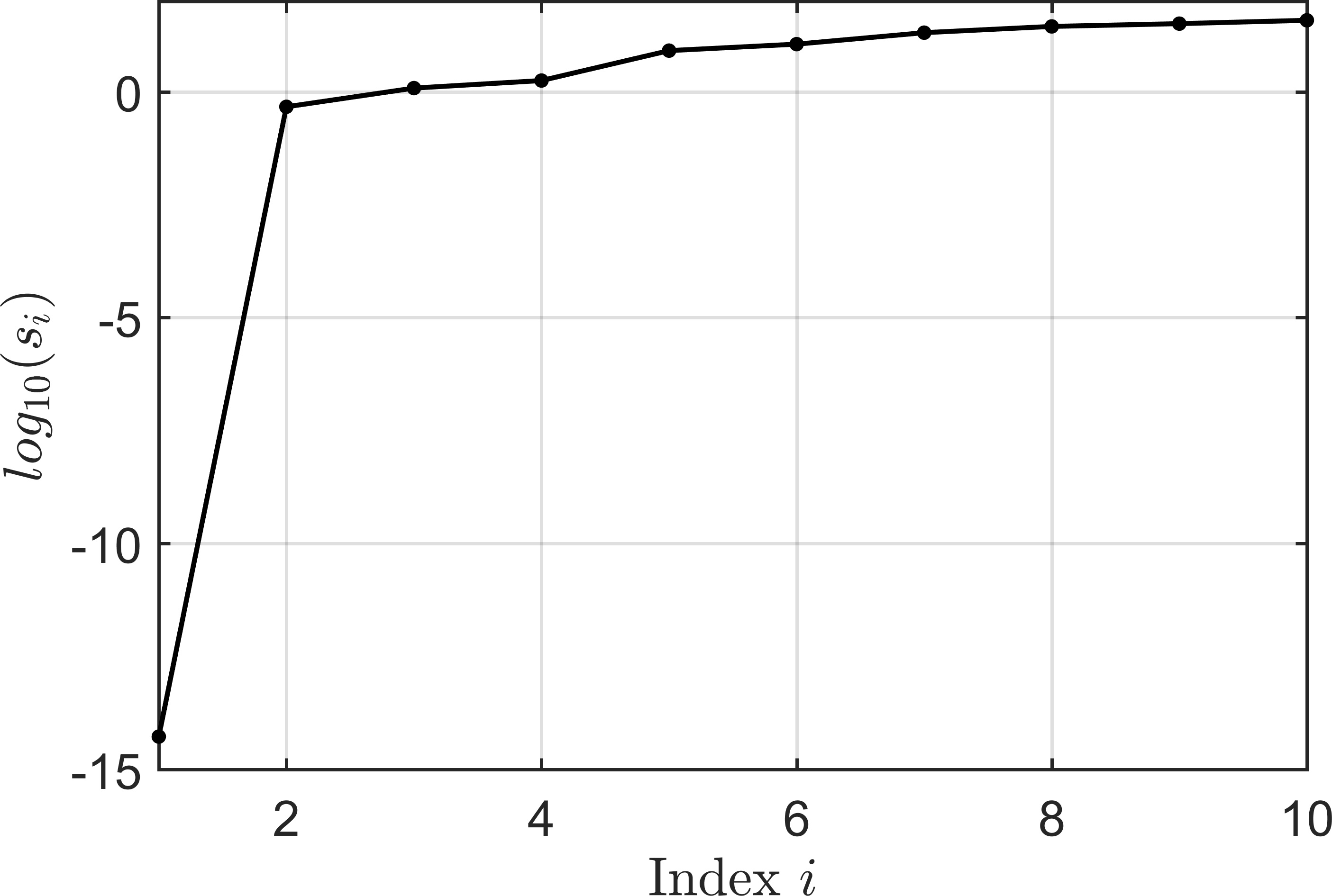}\\[1em]
				\textbf{(b)}
			}
			\parbox[b]{0.32\textwidth}{
				\centering 
				\includegraphics[width=.32\textwidth]{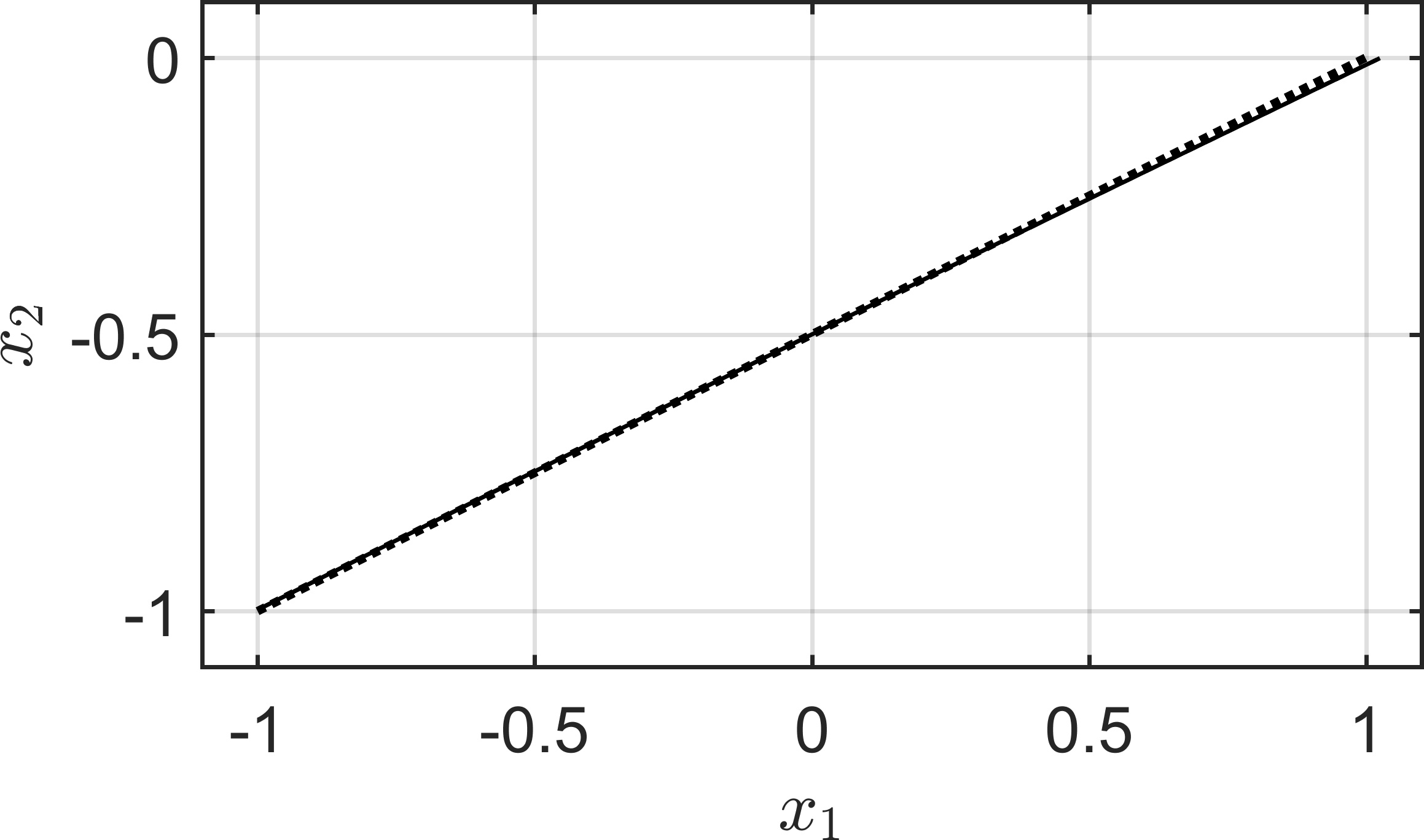}\\[1em]
				\textbf{(c)}
			}
			\caption{\textbf{(a)} Approximation of the solution of \eqref{eq:SMOP} via smoothing Chebyshev scalarization with $1000$ points for Example \ref{example:stochastic}. \textbf{(b)} Singular values of $\L$. \textbf{(c)} Pareto critical set of the original problem (dotted line) and the approximation (solid line).}
			\label{fig:example_stochastic_data}
		\end{figure}
		Since there is no noise in the first component of $f$, the approximation is relatively accurate close to $a$ and becomes worse when moving towards $(1,0)^\top$.
		
		We now interpret the points in Figure \ref{fig:example_stochastic_data}(a) as the data set $\D_x$. The KKT vector $\alpha \in \D_\alpha$ corresponding to $x \in \D_x$ is chosen as $\alpha = (-x_2, 1 + x_2)^\top$. 
		We choose $\B$ as the monomials up to degree $2$, i.e., 
		\begin{equation*}
			\B := \{ x_1, x_1^2, x_2, x_1 x_2, x_2^2 \}.
		\end{equation*}
		In this basis, the objective vector $F$ of \eqref{eq:SMOP} can be represented exactly (up to the constants in both components) by the coefficient vector
		\begin{equation} \label{eq:example_stochastic_exact_coefficients}
			\bar{c} = (2,1,2,0,1,-2,1,0,0,1)^\top.
		\end{equation}
		When applying Algorithm \ref{algo:MOP_from_data}, we obtain the singular values shown in Figure \ref{fig:example_stochastic_data}(b). The objective vector $x \mapsto (2 x_2 + x_2^2, x_2^2)^\top$ corresponding to the smallest singular value $s_1 = 5.33 \cdot 10^{-15}$ is degenerate due to the missing dependency on $x_1$. The next smallest singular values are
		\begin{align*}
			s_2 = 0.4654, \\
			s_3 = 1.2076, \\
			s_4 = 1.7744, \\
			s_5 = 8.1572.
		\end{align*}
		Hence, due to the gap from $s_4$ to $s_5$, we choose $\bar{s} = s_4$, i.e., $I = \{1,2,3,4\}$, in step 3. Calculating a (normalized) sparse basis $\{ w_1, w_2, w_3, w_4 \}$ of $\vspan( \{ v_1, v_2, v_3, v_4 \} )$ results in
		\begin{align*}
			w_1 &= (0,0,-1,0,-0.5,0,0,0,0,-0.5)^\top, \\
			w_2 &= (-0.4995, 0.2465, -0.9955, -0.9944, 0, 0, -0.0009, 0, 0.0031, -1)^\top, \\
			w_3 &= (-1, -0.4989, 0, -0.0029, 0, 0.9948, -0.4844, -0.0004, 0, 0.0019)^\top, \\
			w_4 &= (-0.4954, -0.2478, 0, 0, -0.0024, 0, 0.0027, 0.9945, -0.9837, 1)^\top,
		\end{align*}
		which shows that in step 4, we can choose
		\begin{equation*}
			c^* := -2 w_1 - 2 w_3 = (2, 0.9977, 2, 0.0058, 1, -1.9896, 0.9687, 0.0007, 0, 0.9963)^\top,
		\end{equation*}				
		which is close to $\bar{c}$ (cf. \eqref{eq:example_stochastic_exact_coefficients}). The Pareto critical set of the corresponding objective vector $\F(c^*)$ is shown in Figure \ref{fig:example_stochastic_data}(c). A numerical approximation of the Hausdorff distance between the two sets (using a pointwise discretization) yields $2.5 \cdot 10^{-2}$ (and the corresponding points of maximal distance are located close to $(1,0)^\top$). As functions, comparing $F$ and $\F(c^*)$ (up to constants) around the Pareto critical set yields
		\begin{equation*}
			\max_{x \in [-1.1,1.1] \times [-1.1,0.1]} \| (F(x) - (2, 4/3)^\top) - \F(c^*)(x) \|_\infty \approx 5.46 \cdot 10^{-2},
		\end{equation*}
		showing that we were able to construct a very good approximation of the objective vector $F$ from noisy data.
	\end{example}
	
	\section{Application 2: Generation of surrogate models of expensive MOPs} \label{sec:surrogates}	
	In this section, we will use the results from Section \ref{sec:generating_mops_from_data} for the generation of surrogate models for MOPs with an objective vector $f^e$ that is known but very costly to evaluate. This scenario occurs frequently for complex physics simulations, e.g., when the system under consideration is described by a partial differential equation, cf. \cite{LKBM05,ARFL09,POBD18,BDPV18} for examples. Here, while it is often possible to calculate single Pareto critical points, the computation of the full Pareto critical set via a fine pointwise approximation is computationally infeasible. In this situation, we will use Pareto critical points of the expensive model $f^e$ and their corresponding KKT vectors as data points. Our goal is to find an MOP whose extended Pareto critical set is as close as possible to the extended Pareto critical set of $f^e$ while using as few data points as possible.
	
	Surrogate modeling is a very active area of research and has been used extensively for simulation and optimization, see \cite{SVR05,BGW15} for overviews. In recent years, surrogate models have also attracted interest in the multiobjective optimization community. All methods proposed so far have the common goal of finding a surrogate model for the objective function $f^e$, for instance by polynomial regression (cf., e.g., \cite{CSHM2017,VK2010}). Consequently, the surrogate model will possess a dominance relation similar to the original function and as a result, dominance-based methods like evolutionary algorithms can be applied. In contrast to this, our approach constructs surrogate models which resemble the original KKT conditions. This means that we (in general) do not obtain a surrogate model for the objective function but for the first-order optimality condition such that KKT-based methods like continuation \cite{SDD2005} can be used.
	
	When ``fitting'' a surrogate model to a data set of limited size as in this case, it is important avoid \textit{underfitting} and \textit{overfitting}. These terms are common in statistics and machine learning, but they apply here in a similar fashion. In general, underfitting means that the chosen model is not able to capture all structures that are present in the data set. In our context, this means that we chose an unsuited (e.g., too small) set of basis functions. When using monomials as basis functions, one can try to circumvent this by using a higher maximal degree (as in Example \ref{example:3_con_comp}). On the other hand, overfitting means that the model captures structures in the data set that were caused by noise and are highly dependent on the data used for fitting the model. In our context, this happens when the number of basis functions $d$ in $\B$ is too large. A necessary condition to circumvent this is to ensure that $n \cdot N \geq k \cdot d$, i.e., 
	\begin{equation} \label{eq:condition_overfitting}
		d \leq \frac{n \cdot N}{k}.
	\end{equation}		
	As discussed in Section \ref{sec:generating_mops_from_data}, if this condition does not hold then we always find an objective vector in the chosen basis where the data points are exactly extended Pareto critical. Thus, if \eqref{eq:condition_overfitting} is violated, overfitting is unavoidable (as long as we are not working with exact data).
	
	To illustrate the behavior of our method, we begin with an example where the objective vector is cheap to evaluate and we already know the solution of the MOP. We consider the problem $L \& H_{2 \times 2}$ from \cite{HL2015}, where the objective vector is non-polynomial and has a complex (extended) Pareto critical set.
	
	\begin{example} \label{example:LH_2}
		Consider the MOP
		\begin{align} \label{eq:LH_2}
			&\min_{x \in \R^2} f^e(x) \tag{$L \& H_{2 \times 2}$} \\
			s.t. \quad & x \in [-0.75, 0.75] \times [-2.5,0.12] \nonumber
		\end{align}
		for
		\begin{equation*}
			f^e(x) := - \begin{pmatrix}
			\frac{\sqrt{2}}{2} x_1 +\frac{\sqrt{2}}{2} b(x) \\
			-\frac{\sqrt{2}}{2} x_1 + \frac{\sqrt{2}}{2} b(x)
		\end{pmatrix}
		\end{equation*}
		with
		\begin{align*}
			&b(x) := 0.2 g(x,(0,0)^\top,0.65) + 1.5 g(x,(0,-1.5)^\top,2.8), \\
			&g(x,p_0,\sigma) := \sqrt{\frac{2 \pi}{\sigma}} \exp \left( - \frac{\| x - p_0 \|_2^2}{\sigma^2} \right).		
		\end{align*}
		The Pareto critical set of this MOP and its image are shown in Figures \ref{fig:example_LH_2_exact_data}(a) and (b), respectively.
		\begin{figure}[ht] 
			\parbox[b]{0.32\textwidth}{
				\centering 
				\includegraphics[width=0.32\textwidth]{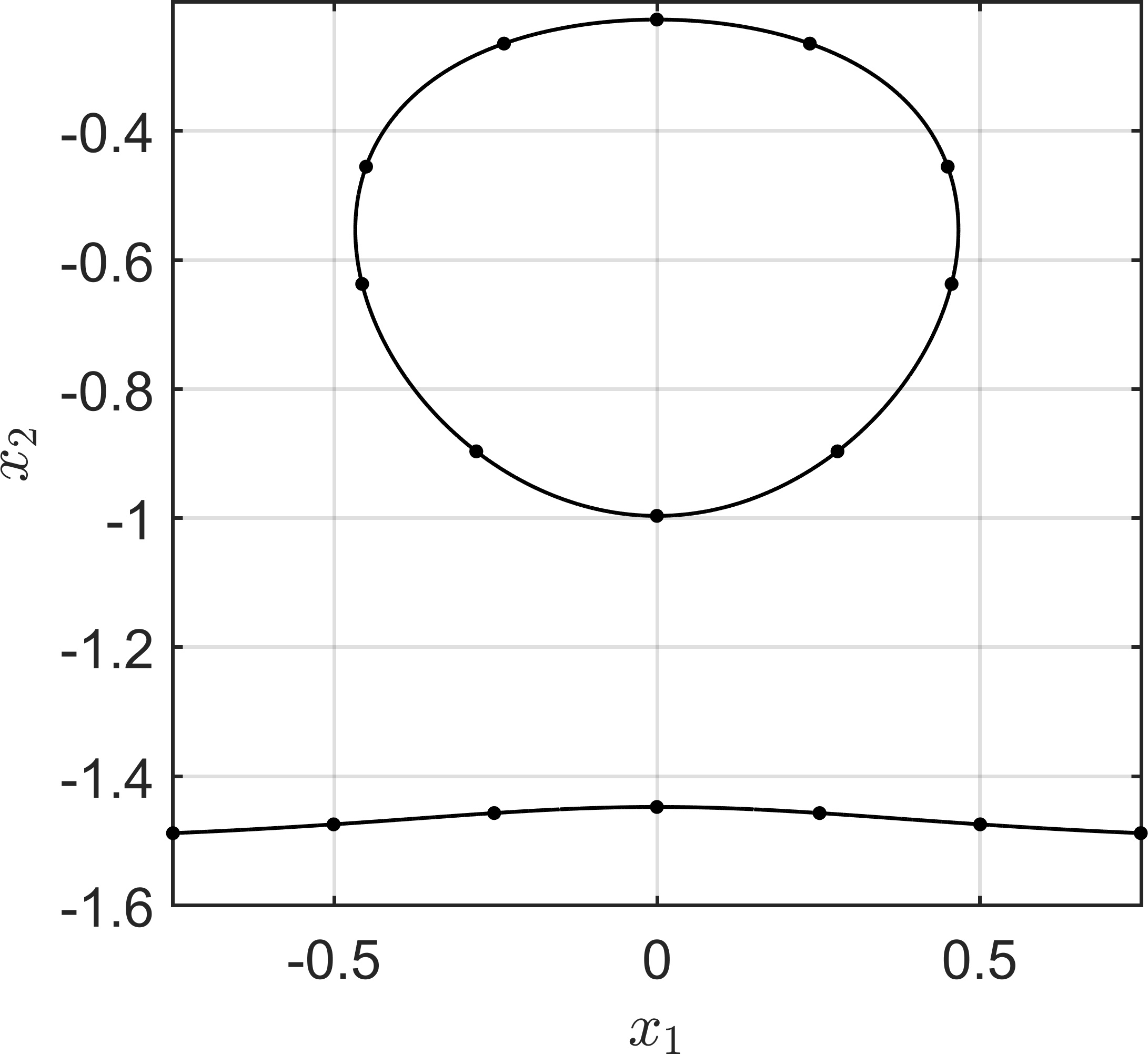}\\[1em]
				\textbf{(a)}
			}
			\parbox[b]{0.32\textwidth}{
				\centering 
				\includegraphics[width=0.32\textwidth]{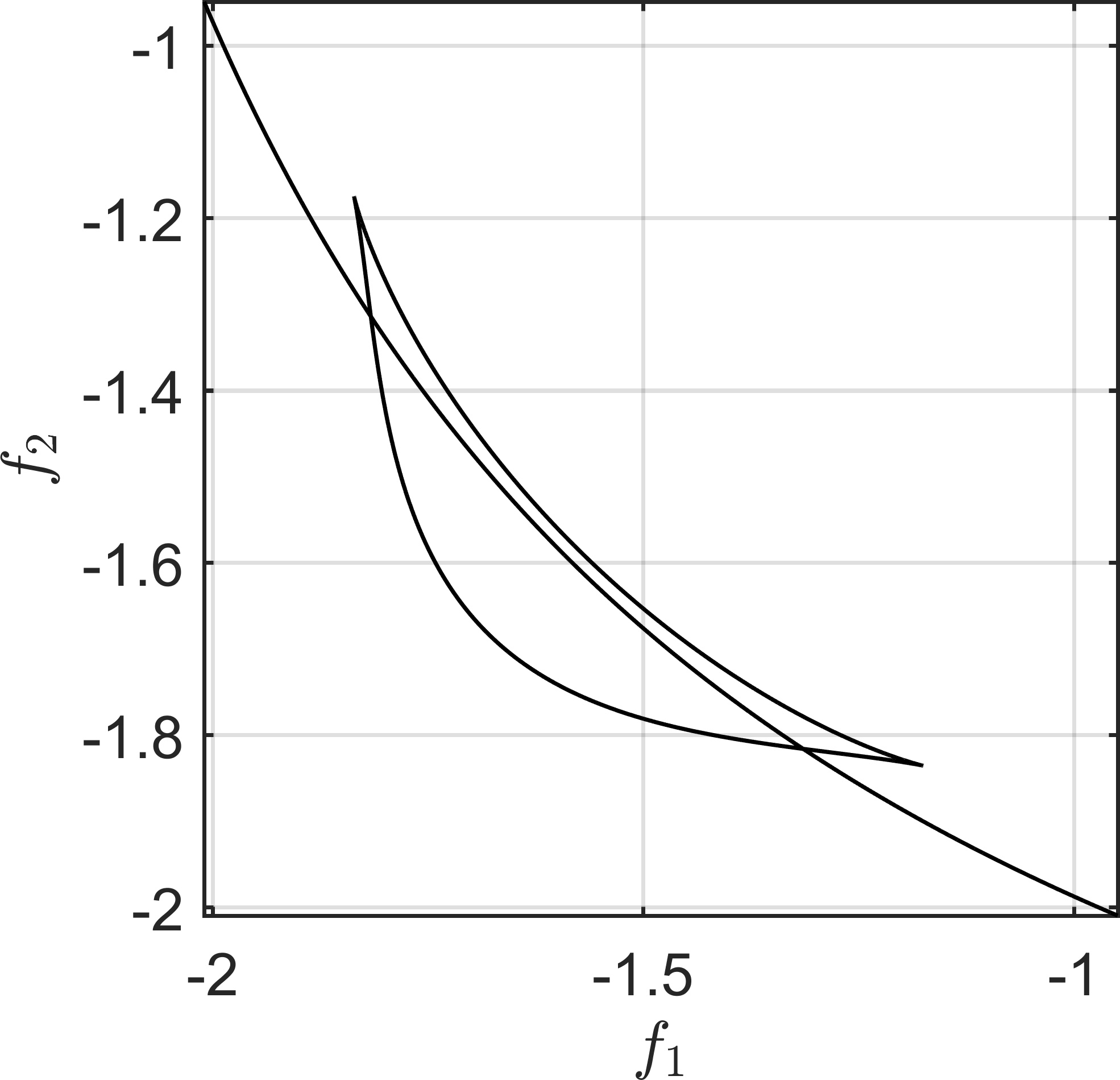}\\[1em]
				\textbf{(b)}
			} \vspace{10pt}
			\parbox[b]{0.32\textwidth}{
				\centering 
				\includegraphics[width=0.32\textwidth]{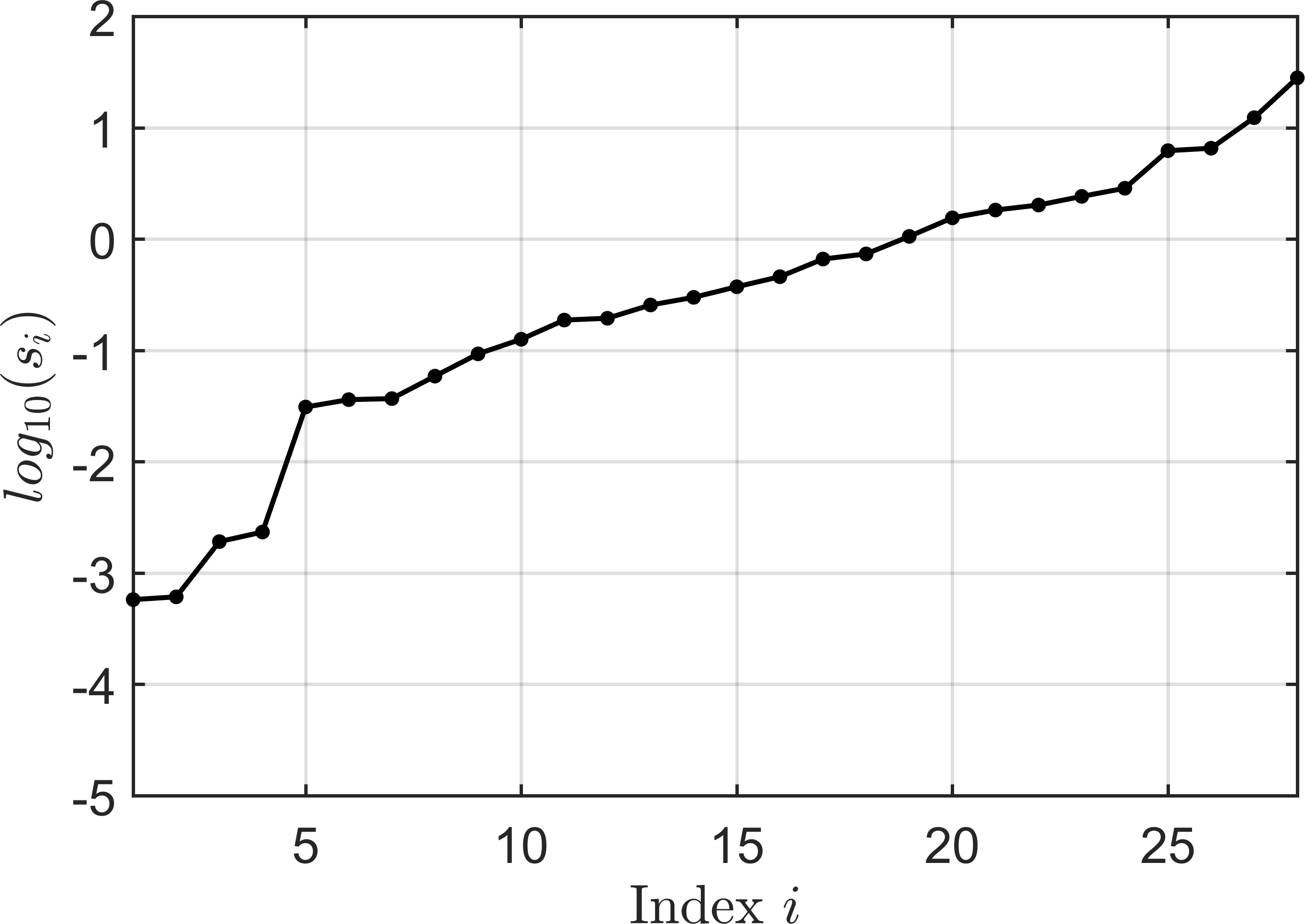}\\[1em]
				\textbf{(c)}
			}
			\caption{\textbf{(a)} Pareto critical set of \eqref{eq:LH_2}. The dots represent the $N = 17$ data points used for the surrogate model construction. \textbf{(b)} Image of the Pareto critical set. \textbf{(c)} Singular values of $\L$ for monomials up to degree $4$ for the chosen data points.}
			\label{fig:example_LH_2_exact_data}
		\end{figure}
		(Note that since all Pareto critical points on the boundary of the feasible set are also Pareto critical for the unconstrained problem, we can consider this problem as unconstrained.) For the surrogate model construction, we choose the $N = 17$ data points depicted in Figure \ref{fig:example_LH_2_exact_data}(a). We choose all monomials up to degree $4$ as basis functions.
		The reason for this choice is that for larger degrees we have $|\B| = d \geq 20$ such that $n \cdot N = 34 < 40 \leq k \cdot d$, which would result in overfitting. 
		
		The surrogate model is now constructed from the data set by applying Algorithm \ref{algo:MOP_from_data}. The singular values of $\L$ are shown in Figure \ref{fig:example_LH_2_exact_data}(c). In steps 3 and 4, we choose the coefficient vector $c = \frac{v_1}{\| v_1 \|_2}$ corresponding to the smallest singular value $s_1 = 5.76 \cdot 10^{-4}$. A comparison between the Pareto critical sets of the corresponding objective vector $f := \F(c)$ and the original objective vector \eqref{eq:LH_2} is shown in Figure \ref{fig:example_LH_2_compare}.
		\begin{figure}[ht] 
			\parbox[b]{0.49\textwidth}{
				\centering 
				\includegraphics[width=0.35\textwidth]{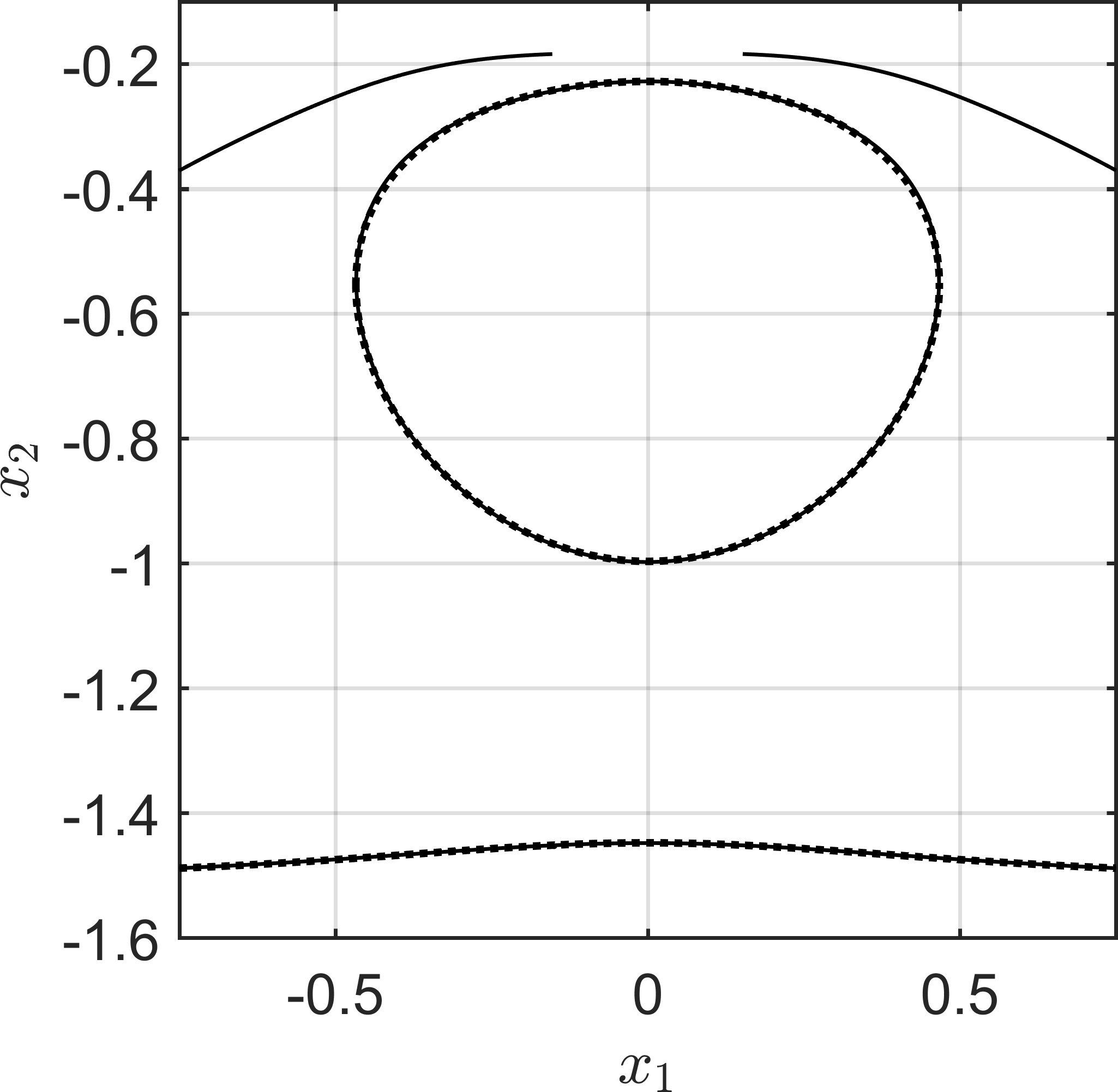}\\
				\textbf{(a)}
			}
			\parbox[b]{0.49\textwidth}{
				\centering 
				\includegraphics[width=0.35\textwidth]{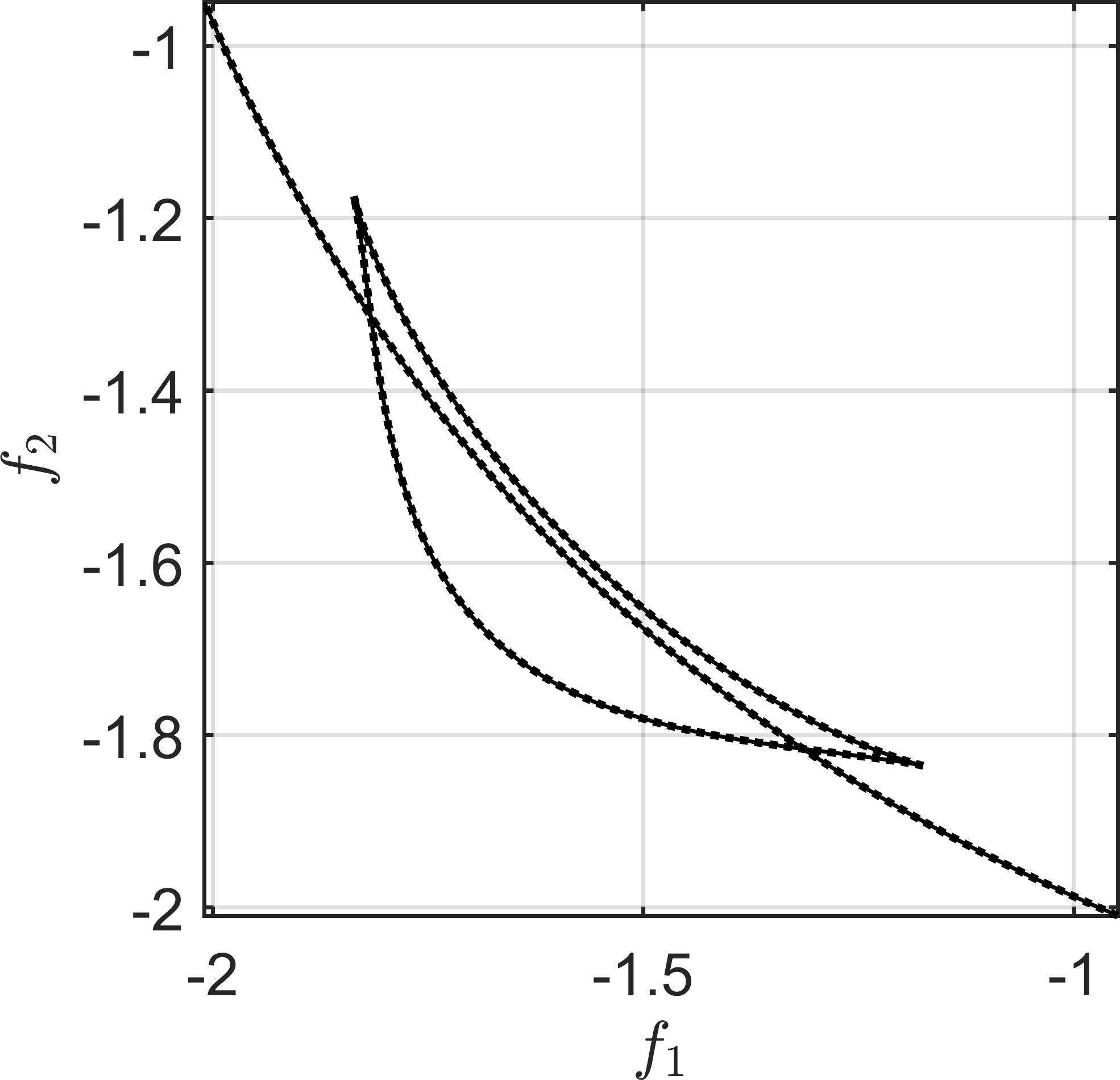}\\
				\textbf{(b)}
			}
			\caption{\textbf{(a)} The Pareto critical sets of \eqref{eq:LH_2} (dotted line) and its approximation $f$ (solid line). \textbf{(b)} The images of the Pareto critical sets of \eqref{eq:LH_2} (dotted line) and $f$ under $f^e$ (solid line).}
			\label{fig:example_LH_2_compare}
		\end{figure}
		We see in (a) that the Pareto critical sets are almost identical besides the two additional connected components at the top. After filtering these out (e.g., by applying clustering algorithms, cf. \cite{S2007}), the Hausdorff distance between the two sets is $4 \cdot 10^{-3}$. Figure \ref{fig:example_LH_2_compare}(b) shows the image of the Pareto critical set of $f$ under the original objective vector $f^e$ without the additional connected components. Similar to the decision space, the Pareto fronts are almost identical with a Hausdorff distance of $1.6 \cdot 10^{-3}$.
	\end{example}
	
	The previous example shows that few data points of the original objective vector can suffice to generate a good surrogate model, even if the original objective vector does not lie in the span of the chosen basis functions. In order to highlight the potential for increased efficiency in real-world applications, our next example considers an MOP where the evaluation of the objective vector is very expensive.
	
	\begin{example} \label{example:expensive}
		In this example, we consider the flow around a cylinder governed by the 2D incompressible Navier--Stokes equations at a Reynolds number of 100, where the goal is to influence the flow field by rotating the cylinder (cf. Figure \ref{fig:example_expensive_WS_result}(a)):
		\begin{equation} \label{eq:NSE} \tag{NSE}
		\begin{aligned}
			\dot{y}(x,t) + y(x,t) \cdot \nabla y(x,t) &= \nabla p(x,t) + \frac{1}{Re} \Delta y(x,t), \\
			\nabla \cdot y(x,t) &= 0, \\
			y(x,0) &= y^0(x).
		\end{aligned}
		\end{equation}
		Here, $y$ is the flow velocity and $p$ is the pressure. For the non-rotating cylinder, the well-known \emph{von K\'{a}rm\'{a}n vortex street} occurs. This is a periodic solution where vortices detach alternatingly from the upper and lower edge of the cylinder, respectively. 
		This setup is a classical problem from flow control which has been studied extensively in the literature both using direct approaches as well as surrogate models, see \cite{POBD18} and the references therein. The classical goal is to stabilize the flow, i.e., to minimize the vertical velocity. This can be associated with minimizing the vertical force on the cylinder, the \emph{lift} $C_L$. As a second goal, we want to minimize the control effort, which results in the following multiobjective optimal control problem:
		\begin{equation} \label{eq:example_expensive_MOP1}
			\begin{aligned}
			\min_{u \in L^2([t_0,t_e], \R)} &\left( \begin{array}{c} \int_{t_0}^{t_e} C_L^2(t) \, dt \\ \int_{t_0}^{t_e} u^2(t) \, dt \end{array} \right) \\
			\mbox{s.t.} \qquad &\eqref{eq:NSE}.
			\end{aligned}
		\end{equation}
		By introducing a sinusoidal control $u(t) = x_1 sin(2 \pi x_2\, t)$ and assuming that the control-to-state mapping is injective, Problem \eqref{eq:example_expensive_MOP1} can be transformed into the MOP
		\begin{equation} \label{eq:example_expensive_MOP}
			\min_{x \in \R^2} f^e(x) \ \text{ with } \
			f^e(x) := \begin{pmatrix}
					\int_{t_0}^{t_e} C_L^2(t) \, dt \\
					\int_{t_0}^{t_e} (x_1 sin(2 \pi x_2\, t))^2 \, dt 
				\end{pmatrix}.
		\end{equation}	
		Since the Navier--Stokes equations are a system of nonlinear partial differential equations, we have to introduce a spatial discretization (here via the finite volume method) with $22,000$ cells, which results in $66,000$ degrees of freedom at each time instant. Consequently, it is infeasible to accurately solve Problem \eqref{eq:example_expensive_MOP} directly, regardless of the used method. 
	
		One way to approach this problem is to introduce a surrogate model for the system dynamics \eqref{eq:NSE}, for instance via \emph{Proper Orthogonal Decomposition} \cite{POBD18}. In contrast to this, here, we directly construct a surrogate model for the MOP \eqref{eq:example_expensive_MOP} instead of the system dynamics. In order to generate the required data points $\D$, we apply scalarization via the well-known Weighting Method (i.e., $\min w_1 f^e_1 + w_2 f^e_2$, cf. \cite{M1998}) to \eqref{eq:example_expensive_MOP} with varying weights
		\begin{equation} \label{eq:example_expensive_weights}
			w^i = \left( \frac{i - 1}{25}, 1 - \frac{i - 1}{25} \right)^\top, \quad i \in \{ 1, ..., 26 \}.
		\end{equation}				
		An advantage of this method is that we directly obtain the KKT vectors of the resulting Pareto optimal points as the corresponding weights that were used to calculate them. Since there are convergence issues for $i \in \{10,...,16\}$ using the weighted sum (likely due to a large number of local minima, which is a known problem), we will exclude these points from our data set. The remaining $19$ points are shown in Figures \ref{fig:example_expensive_WS_result}(b) and (c).
		\begin{figure}[ht] 
			\parbox[b]{0.32\textwidth}{
				\centering 
				\includegraphics[width=0.32\textwidth]{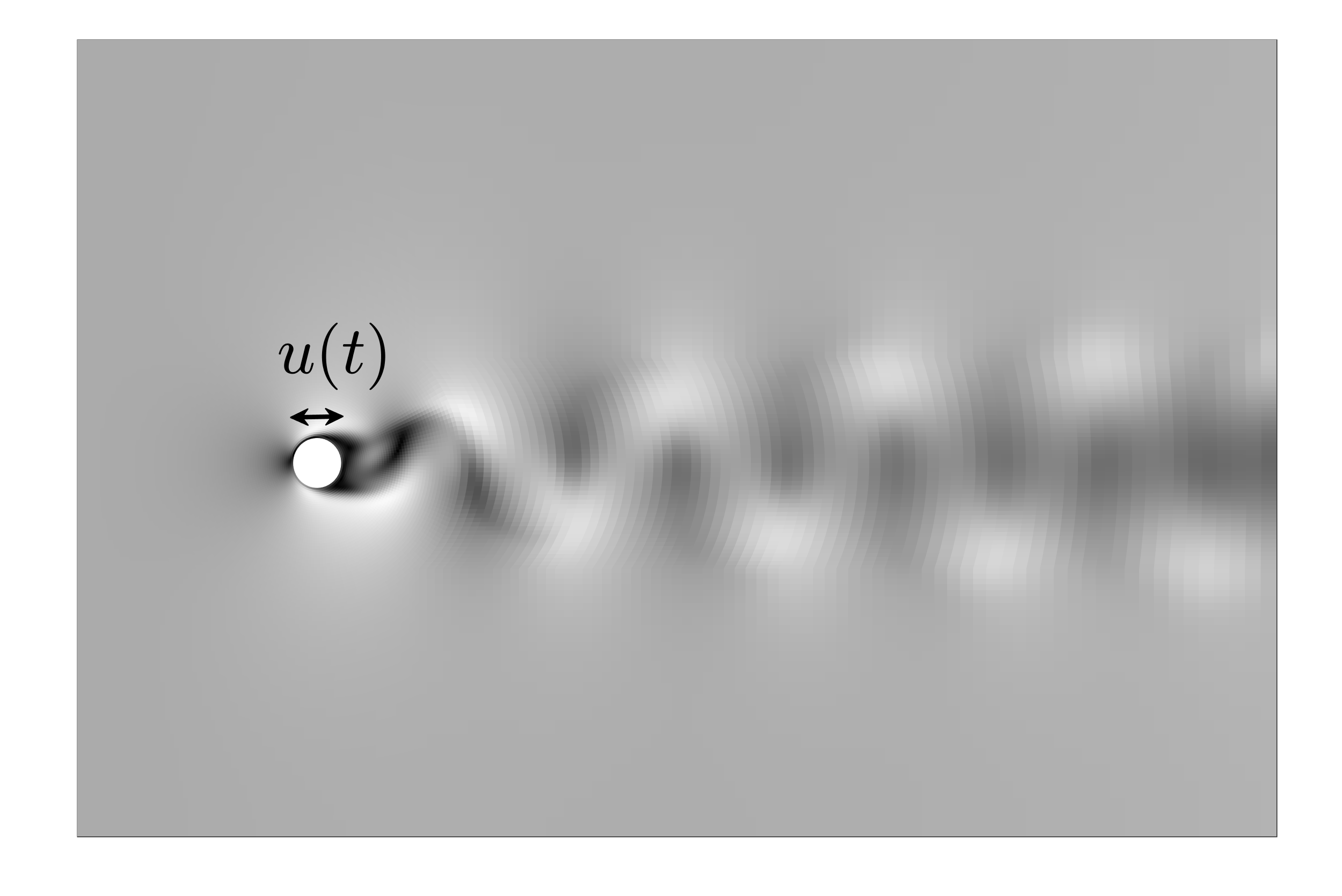}\\[1em]
				\textbf{(a)}
			}
			\parbox[b]{0.32\textwidth}{
				\centering 
				\includegraphics[width=0.32\textwidth]{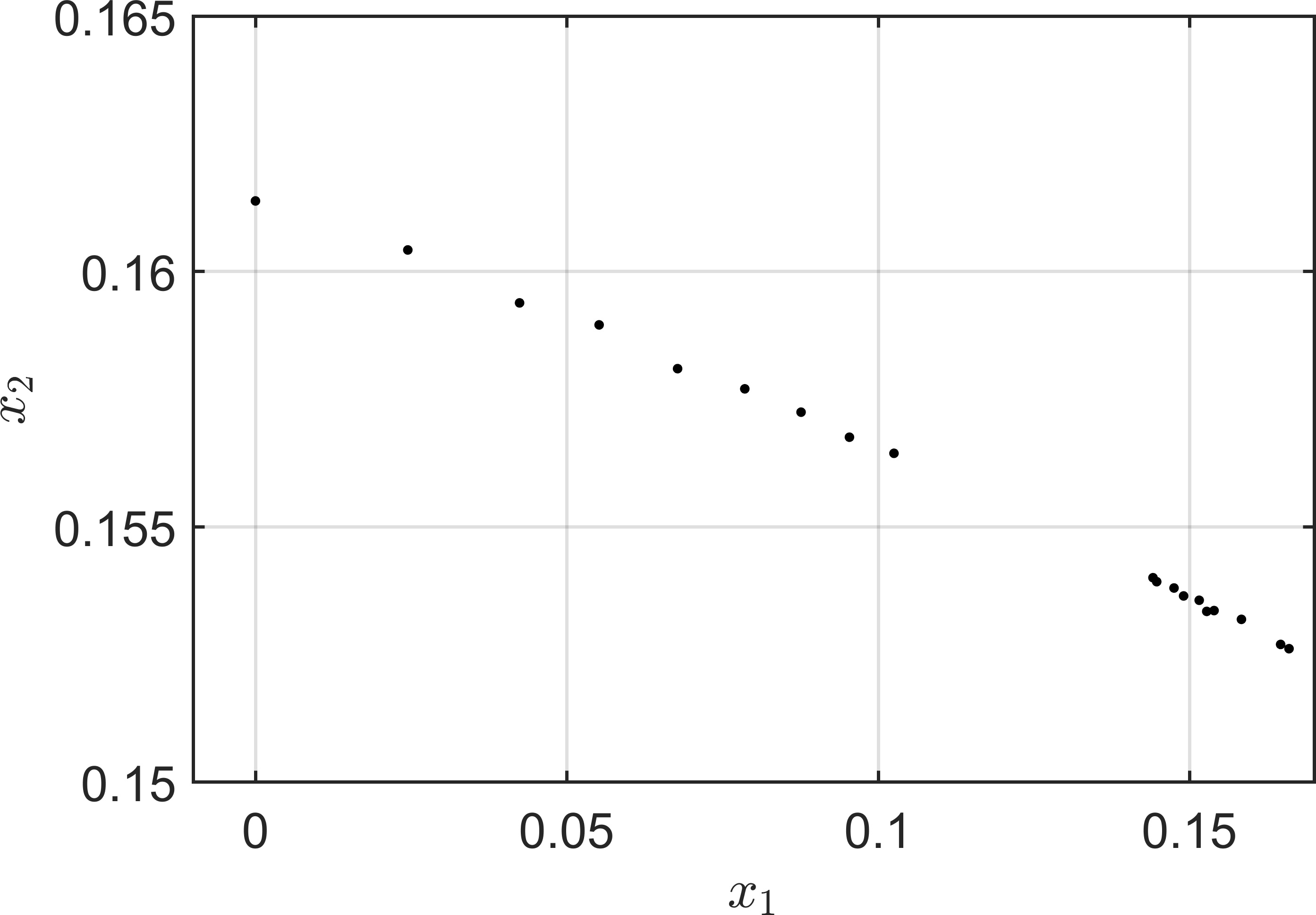}\\[1em]
				\textbf{(b)}
			}
			\parbox[b]{0.32\textwidth}{
				\centering 
				\includegraphics[width=0.32\textwidth]{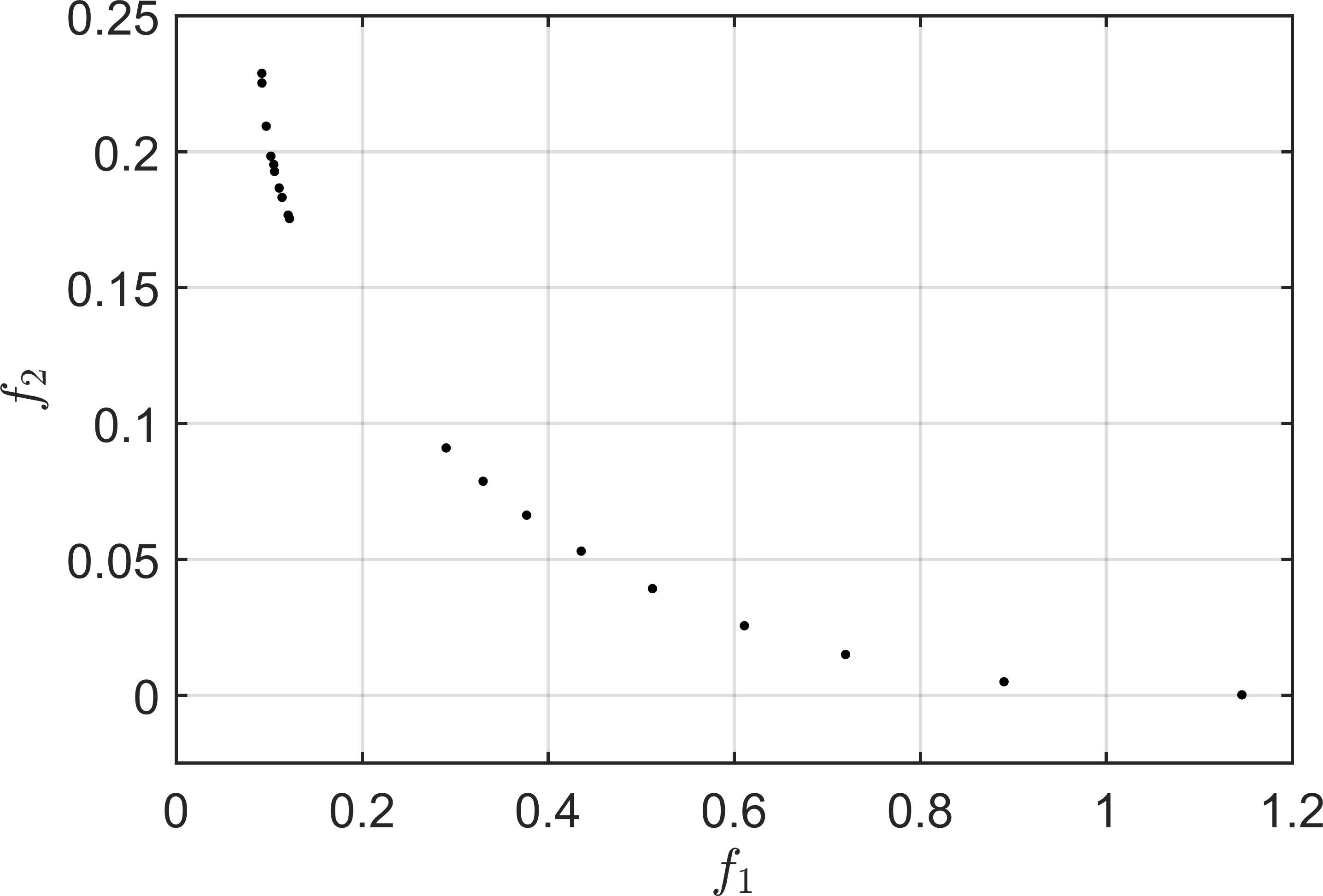}\\[1em]
				\textbf{(c)}
			}
			\caption{\textbf{(a)} Flow around a cylinder, controlled via cylinder rotation. \textbf{(b)} Result of the Weighting Method applied to the MOP \eqref{eq:example_expensive_MOP} in the variable space. \textbf{(c)} Image of the resulting points under the objective function \eqref{eq:example_expensive_MOP}.}
			\label{fig:example_expensive_WS_result}
		\end{figure}
	
		Considering $\D_x$ and $\D_\alpha$, it appears that the Pareto set consists of a single one-dimensional connected component whose corresponding KKT vectors are monotonically increasing and decreasing in their first and second component, respectively. Due to this simple structure, we take all monomials up to degree $2$ as our set of basis functions. The singular values of the resulting $\L \in \R^{38 \times 10}$ are shown in Figure \ref{fig:example_expensive_singular_values}(a).	
		\begin{figure} 
			\parbox[b]{0.49\textwidth}{
				\centering
				\includegraphics[width=.4\textwidth]{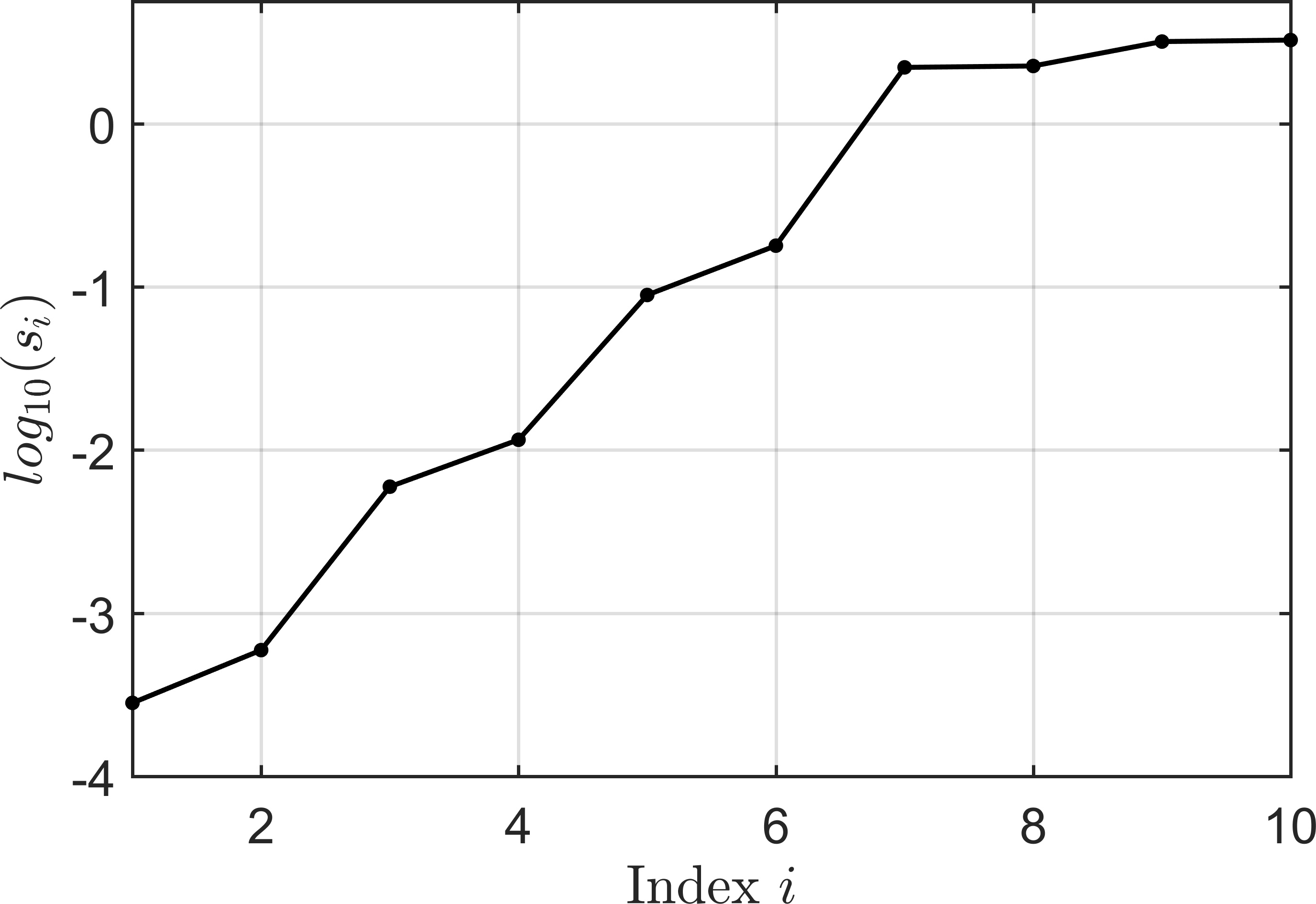}\\[1em]
				\textbf{(a)}
			}
			\parbox[b]{0.49\textwidth}{
				\centering 
				\includegraphics[width=0.4\textwidth]{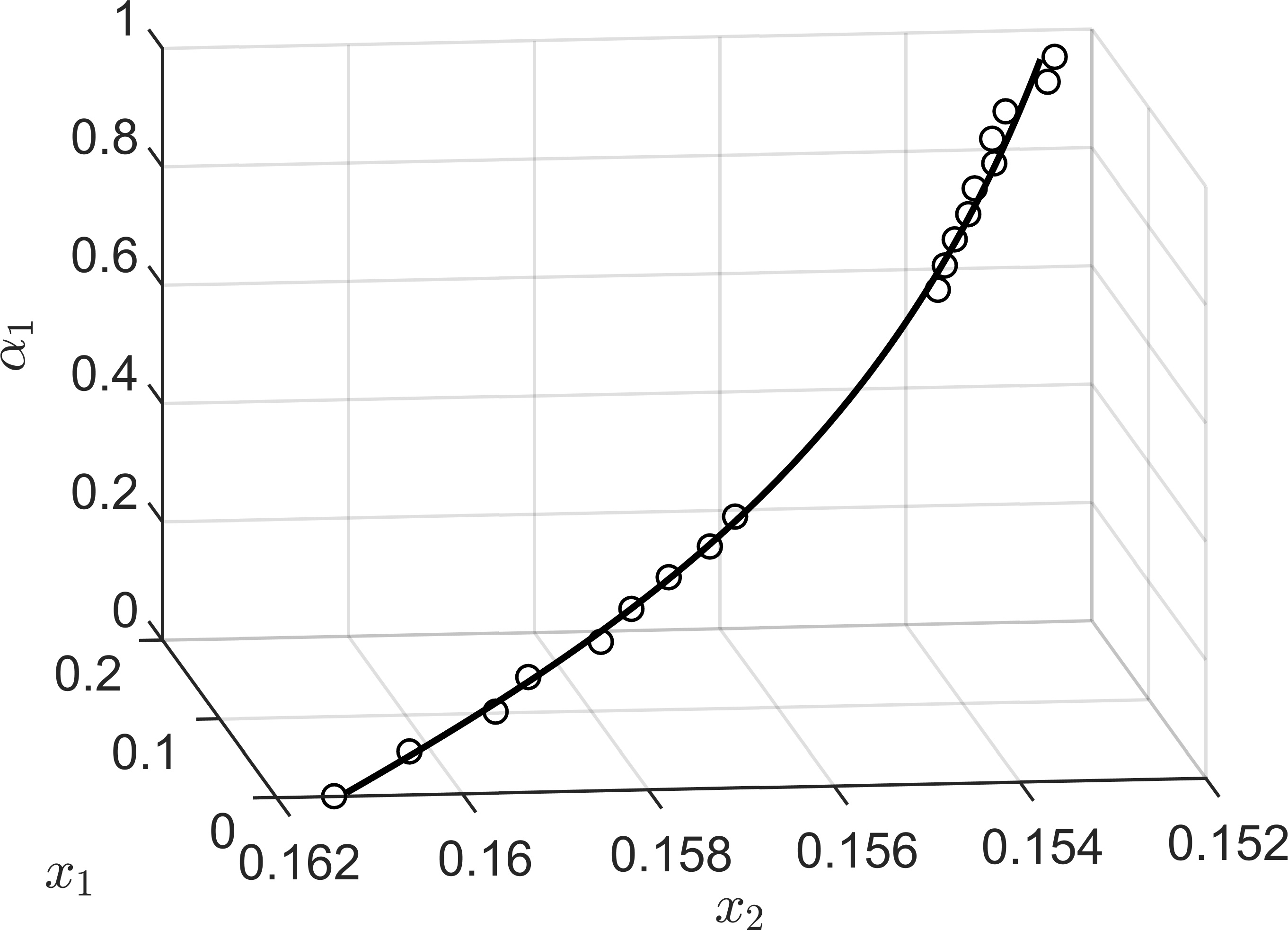}\\[1em]
				\textbf{(b)}
			}
			\caption{\textbf{(a)} Singular values of $\L$ in Example \ref{example:expensive}. \textbf{(b)} Pareto critical set of \eqref{eq:example_expensive_surrogate_model} (solid line) and the data in $\D_x$ (circles), where the first KKT multiplier $\alpha_1$ is shown in the third dimension.}
			\label{fig:example_expensive_singular_values}
		\end{figure}
		The smallest singular values $s_1 = 2.82 \cdot 10^{-4}$ and $s_2 = 5.95 \cdot 10^{-4}$ correspond to objective vectors where the influence of $x_1$ is relatively small. (In particular, the hessian matrices of both objective functions in both objective vectors are almost singular.) Therefore, the corresponding Pareto critical sets are degenerate similar to the objective vector \eqref{eq:example_circle_degenerated_objective} in Example \ref{example:unit_circle}. Due to this, we instead consider the objective vector corresponding to the third singular value $s_3 = 5.96 \cdot 10^{-3}$ as our surrogate model, given by
		\begin{equation} \label{eq:example_expensive_surrogate_model}
			f(x) = \begin{pmatrix}
			- 0.0519 x_1^2 - 0.9285 x_1 x_2 + 0.1588 x_1 + 0.1542 x_2^2 + 0.1046 x_2 \\
			- 0.0136 x_1^2 - 0.2704 x_1 x_2 + 0.0437 x_1 + 0.0054 x_2^2 - 0.0008 x_2
			\end{pmatrix}.
		\end{equation}
		A projection of the corresponding extended Pareto critical set is depicted in Figure \ref{fig:example_expensive_singular_values}(b), showing that all data points are close to the solution of the surrogate problem. In order to obtain an approximation of the Pareto front of the original MOP \eqref{eq:example_expensive_MOP}, we can evaluate the original objective vector $f^e$ in a pointwise discretization of the Pareto critical set of the surrogate model $f$. In order to evaluate the performance, we compare our results with the well-known NSGA-II algorithm \cite{DPAM02} (implementation from MATLAB's Global Optimization Toolbox) directly applied to the MOP \eqref{eq:example_expensive_MOP}. The results are depicted in Figure \ref{fig:example_expensive_compare_result}.
		\begin{figure}[ht] 
			\parbox[b]{0.49\textwidth}{
				\centering 
				\includegraphics[width=0.4\textwidth]{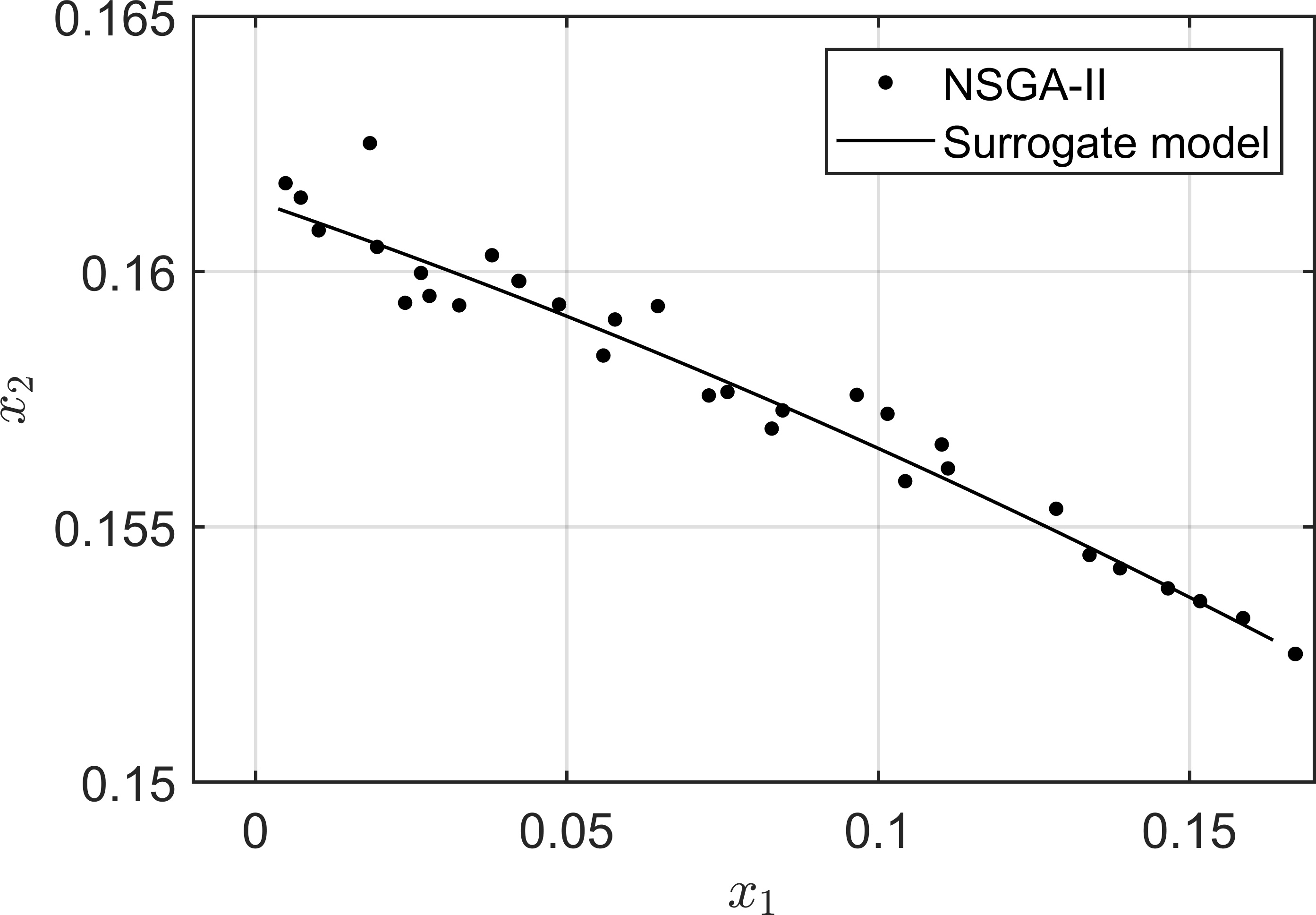}\\
				\textbf{(a)}
			}
			\parbox[b]{0.49\textwidth}{
				\centering 
				\includegraphics[width=0.4\textwidth]{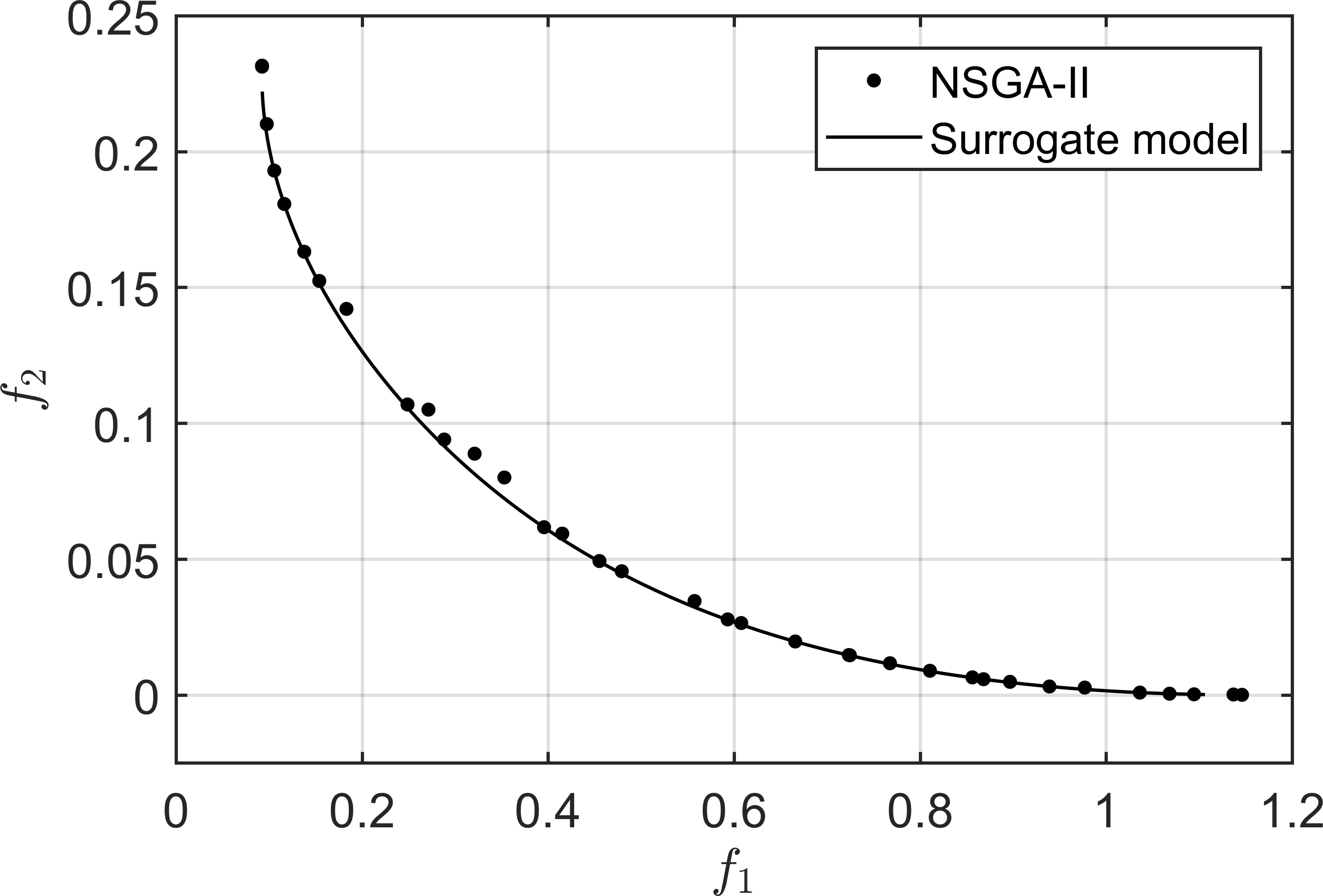}\\
				\textbf{(b)}
			}
			\caption{\textbf{(a)} The approximation of the Pareto set via the surrogate model compared to NSGA-II in Example \ref{example:expensive}. \textbf{(b)} Comparison of the corresponding approximations of the Pareto fronts.}
			\label{fig:example_expensive_compare_result}
		\end{figure}
		Here, we have used an initial population size of $100$ for NSGA-II and a discretization of the Pareto critical set of our surrogate model \eqref{eq:example_expensive_surrogate_model} with $468$ equidistant points. Figure \ref{fig:example_expensive_compare_result} shows that although we only used $19$ data points for the generation of our surrogate model and there was a gap in our data set, we are able to obtain a good approximation of the Pareto set and front in a very efficient manner.
	\end{example}
	
	As already mentioned, it is crucial for our approach to obtain not only Pareto critical points of the given, expensive MOP, but also the corresponding KKT vectors. As mentioned in the previous example, when applying the Weighting Method, the KKT vector of an optimal point is immediately given. This means that no additional effort has to be put into computing the KKT vectors in that case. Similar results can be shown for the $\epsilon$-Constraint Method and the Reference Point Method (cf.~\cite{M1998}), where the KKT vectors can be obtained from the first order optimality condition of the scalar subproblems. For other methods, in particular evolutionary algorithms, we cannot always expect to directly obtain the KKT vectors for free. Given only the Pareto critical (or optimal) point, a straight-forward way to obtain the corresponding KKT vector would be to evaluate the gradients of the objective functions and solve \eqref{eq:KKT} as a linear system in $\alpha$. However, this approach can obviously be very time consuming. Furthermore, knowledge of the derivatives is required. A much cheaper alternative is to exploit the fact that KKT vectors are orthogonal to the linearized Pareto front \cite{H2001}. For a pointwise approximation of the Pareto front, e.g., obtained by NSGA-II, we can use linear regression in each point of the front using only the neighboring points on the front to obtain an approximation of the tangent space of the Pareto front. While this requires a relatively even discretization of the Pareto front, it is much cheaper than assembling and solving the above-mentioned linear system.
	
	\section{Conclusion and outlook} \label{sec:conclusion}
	In this article, we present a way to construct objective vectors of MOPs so that the extended Pareto critical sets contain a given data set. This is realized by considering the $x^*$ and $\alpha^*$ in the KKT conditions \eqref{eq:KKT} as given by the data and then searching for an objective vector $f \in C^1(\R^n,\R^k)$ that satisfies the resulting system of equations. By using a finite set of basis functions $\B \subseteq C^1(\R^n,\R)$, the optimal objective vector can be obtained via singular value decomposition, which results in Algorithm \ref{algo:MOP_from_data}. 
	
	The ability to infer objective vectors from (potentially noisy) data has several powerful applications. In examples, we showed how it can be used to generate test problems for solution methods of MOPs and to approximate the Pareto set and objective vector of stochastic MOPs. Alternatively, the approach can be used to significantly reduce the computational effort for expensive MOPs. Using several data points from the expensive problem, a much cheaper surrogate model can be constructed which can be solved significantly faster.
	
	To the best of the authors' knowledge, this article presents the first approach to deal with the inverse problem of multiobjective optimization (i.e., finding the objective vector to a given Pareto (critical) set) in such a general way. Therefore, there are many aspects that are should be investigated further:
	\begin{itemize}
		\item While the results in Section \ref{sec:generating_mops_from_data} hold for any number of variables $n$ and any number of objectives $k$, all the examples in this article consider the case of $k = 2$ objective functions in $n = 2$ variables. This allowed us to easily visualize the most important features of our approach. Nevertheless, the behavior for higher-dimensional examples is worth investigating. 
		\item For the reasons mentioned at the end of Section \ref{sec:generating_mops_from_data}, we only used monomials up to different maximal degrees as basis functions $\B$. Although this lead to satisfactory results in the examples considered here, there might be more sophisticated choices, in particular if one has some knowledge of the problem structure.
		\item As the KKT conditions \eqref{eq:KKT} can also be formulated for equality and inequality constrained MOPs, we expect that a generalization of our approach to constrained MOPs is possible.
		\item Since we can currently only influence the Pareto critical set of the resulting objective vector, an extension to Pareto optimality would be of significant interest, in particular in applications. As sufficient optimality conditions for MOPs are using second order derivatives (cf.~\cite{M1998}), a possible way to control the optimality of the data set might by to incorporate the hessians of the basis functions in our approach.
		\item For the generation of surrogate models, it is important to ensure that the (extended) Pareto critical set of the surrogate model is indeed a good approximation of the actual (extended) Pareto critical set. The convergence result in Theorem \ref{theorem:convergence} states that the smallest singular value of $\L$ is an upper bound for the Euclidean norm of the KKT conditions in the data points. However, this can not be used directly to obtain an estimate for the Hausdorff distance between the Pareto critical set and its surrogate approximation, which is why further investigation of the convergence theory is required.
		\item In order to improve the robustness, it is advisable to develop automated procedures for selecting $\overline{s}$ as well as $c$ (i.e., the threshold for the singular values and the coefficients for the basis).
	\end{itemize}
		
	\noindent
	\textbf{Acknowledgements:} This research was funded by the DFG Priority Programme 1962 ``Non-smooth and Complementarity-based Distributed Parameter Systems''.
		
	\bibliography{literature}   

\begin{thebibliography}{10}

\bibitem{AO2001}
R.~K. Ahuja and J.~B. Orlin.
\newblock {Inverse Optimization}.
\newblock {\em Operations Research}, 49(5):771--783, 2001.

\bibitem{ARFL09}
M.~N. Albunni, V.~Rischmuller, T.~Fritzsche, and B.~Lohmann.
\newblock {Multiobjective Optimization of the Design of Nonlinear
  Electromagnetic Systems Using Parametric Reduced Order Models}.
\newblock {\em IEEE Transactions on Magnetics}, 45(3):1474--1477, 2009.

\bibitem{BDPV18}
D.~Beermann, M.~Dellnitz, S.~Peitz, and S.~Volkwein.
\newblock {Set-Oriented Multiobjective Optimal Control of PDEs using Proper
  Orthogonal Decomposition}.
\newblock In {\em Reduced-Order Modeling (ROM) for Simulation and
  Optimization}, pages 47--72. Springer, 2018.

\bibitem{BGW15}
P.~Benner, S.~Gugercin, and K.~Willcox.
\newblock {A Survey of Projection-Based Model Reduction Methods for Parametric
  Dynamical Systems}.
\newblock {\em SIAM Review}, 57(4):483--531, 2015.

\bibitem{BPK2016}
S.~L. Brunton, J.~L. Proctor, and J.~N. Kutz.
\newblock {Discovering governing equations from data by sparse identification
  of nonlinear dynamical systems}.
\newblock {\em Proceedings of the National Academy of Sciences}, page
  201517384, 2016.

\bibitem{CCLS2014}
T.~C.~Y. Chan, T.~Craig, T.~Lee, and M.~B. Sharpe.
\newblock {Generalized Inverse Multiobjective Optimization with Application to
  Cancer Therapy}.
\newblock {\em Operations Research}, 62(3):680--695, 2014.

\bibitem{CL2017}
T.~C.~Y. Chan and T.~Lee.
\newblock {Trade-off preservation in inverse multi-objective convex
  optimization}.
\newblock {\em European Journal of Operational Research}, 270(1):25 -- 39,
  2018.

\bibitem{CSHM2017}
T.~Chugh, K.~Sindhya, J.~Hakanen, and K.~Miettinen.
\newblock A survey on handling computationally expensive multiobjective
  optimization problems with evolutionary algorithms.
\newblock {\em Soft Computing}, Dec 2017.

\bibitem{dM1976}
W.~De~Melo.
\newblock {On the structure of the Pareto set of generic mappings}.
\newblock {\em Bulletin of the Brazilian Mathematical Society}, 7(2):121--126,
  1976.

\bibitem{D1999}
K.~Deb.
\newblock {Multi-objective Genetic Algorithms: Problem Difficulties and
  Construction of Test Problems}.
\newblock {\em Evolutionary Computation}, 7(3):205--230, Sep. 1999.

\bibitem{DPAM02}
K.~Deb, A.~Pratap, S.~Agarwal, and T.~Meyarivan.
\newblock {A Fast and Elitist Multiobjective Genetic Algorithm: NSGA-II}.
\newblock {\em IEEE Transactions on Evolutionary Computation}, 6(2):182--197,
  2002.

\bibitem{E2005}
M.~Ehrgott.
\newblock {\em Multicriteria {O}ptimization}.
\newblock Springer-Verlag Berlin Heidelberg, 2005.

\bibitem{FX2011}
J.~Fliege and H.~Xu.
\newblock {Stochastic Multiobjective Optimization: Sample Average Approximation
  and Applications}.
\newblock {\em Journal of Optimization Theory and Applications},
  151(1):135--162, Oct 2011.

\bibitem{GGK2014}
W.~Gander, M.~Gander, and F.~Kwok.
\newblock {\em {Scientific Computing - An Introduction using Maple and
  MATLAB}}.
\newblock Springer International Publishing, 2014.

\bibitem{GPD2019}
B.~Gebken, S.~Peitz, and M.~Dellnitz.
\newblock On the hierarchical structure of pareto critical sets.
\newblock {\em Journal of Global Optimization}, 2019.

\bibitem{HL2015}
M.~E. Hartikainen and A.~Lovison.
\newblock {PAINT---SiCon: Constructing Consistent Parametric Representations of
  Pareto Sets in Nonconvex Multiobjective Optimization}.
\newblock {\em Journal of Global Optimization}, 62(2):243--261, June 2015.

\bibitem{H2004}
C.~Heuberger.
\newblock {Inverse Combinatorial Optimization: A Survey on Problems, Methods,
  and Results}.
\newblock {\em Journal of Combinatorial Optimization}, 8(3):329--361, Sep 2004.

\bibitem{H2001}
C.~Hillermeier.
\newblock {\em Nonlinear {M}ultiobjective {O}ptimization: A Generalized
  Homotopy Approach}.
\newblock Birkhäuser Basel, 2001.

\bibitem{HL1999}
S.~Huang and Z.~Liu.
\newblock {On the inverse problem of linear programming and its application to
  minimum weight perfect k-matching}.
\newblock {\em European Journal of Operational Research}, 112(2):421 -- 426,
  1999.

\bibitem{KWP2016}
P.~Kerschke, H.~Wang, M.~Preuss, C.~Grimme, A.~Deutz, H.~Trautmann, and
  M.~Emmerich.
\newblock Towards analyzing multimodality of continuous multiobjective
  landscapes.
\newblock In {\em Parallel Problem Solving from Nature -- PPSN XIV}, pages
  962--972, Cham, 2016. Springer International Publishing.

\bibitem{KWB2011}
A.~Keshavarz, Y.~Wang, and S.~Boyd.
\newblock {Imputing a convex objective function}.
\newblock In {\em 2011 IEEE International Symposium on Intelligent Control},
  pages 613--619, Sept 2011.

\bibitem{KT1951}
H.~W. Kuhn and A.~W. Tucker.
\newblock {Nonlinear Programming}.
\newblock In {\em {Proceedings of the Second Berkeley Symposium on Mathematical
  Statistics and Probability}}, pages 481--492. University of California Press,
  1951.

\bibitem{LKBM05}
A.~V. Lotov, G.~K. Kamenev, V.~E. Berezkin, and K.~Miettinen.
\newblock {Optimal control of cooling process in continuous casting of steel
  using a visualization-based multi-criteria approach}.
\newblock {\em Applied Mathematical Modelling}, 29:653--672, 2005.

\bibitem{LP2014}
A.~Lovison and F.~Pecci.
\newblock {Hierarchical stratification of Pareto sets}.
\newblock {\em arXiv:1407.1755}, 2014.

\bibitem{M1998}
K.~Miettinen.
\newblock {\em Nonlinear {M}ultiobjective {O}ptimization}.
\newblock Springer US, 1998.

\bibitem{POBD18}
S.~Peitz, S.~Ober-Bl{\"{o}}baum, and M.~Dellnitz.
\newblock {Multiobjective Optimal Control Methods for the Navier-Stokes
  Equations Using Reduced Order Modeling}.
\newblock {\em Acta Applicandae Mathematicae}, 2018.

\bibitem{R1991}
W.~Rudin.
\newblock {\em Functional analysis}.
\newblock {New York: McGraw-Hill}, 2nd edition, 1991.

\bibitem{S2007}
S.~E. Schaeffer.
\newblock Graph clustering.
\newblock {\em Computer science review}, 1(1):27--64, 2007.

\bibitem{SVR05}
W.~Schilders, H.~A. van~der Vorst, and J.~Rommes, editors.
\newblock {\em {Model Order Reduction}}.
\newblock Springer Berlin Heidelberg, 2008.

\bibitem{SDD2005}
O.~Sch{\"u}tze, A.~Dell'Aere, and M.~Dellnitz.
\newblock {On Continuation Methods for the Numerical Treatment of
  Multi-Objective Optimization Problems}.
\newblock In J.~Branke, K.~Deb, K.~Miettinen, and R.~E. Steuer, editors, {\em
  Practical Approaches to Multi-Objective Optimization}, number 04461 in
  Dagstuhl Seminar Proceedings, Dagstuhl, Germany, 2005. Internationales
  Begegnungs- und Forschungszentrum f{\"u}r Informatik (IBFI), Schloss
  Dagstuhl, Germany.

\bibitem{T1996}
R.~Tibshirani.
\newblock {Regression Shrinkage and Selection via the Lasso}.
\newblock {\em Journal of the Royal Statistical Society. Series B
  (Methodological)}, 58(1):267--288, 1996.

\bibitem{VK2010}
I.~Voutchkov and A.~Keane.
\newblock {\em Multi-Objective Optimization Using Surrogates}, pages 155--175.
\newblock Springer Berlin Heidelberg, Berlin, Heidelberg, 2010.

\bibitem{ZZZ2008}
Q.~Zhang, A.~Zhou, S.~Zhao, P.~Suganthan, W.~Liu, and S.~Tiwari.
\newblock {Multiobjective optimization Test Instances for the CEC 2009 Special
  Session and Competition}.
\newblock {\em Mechanical Engineering}, Jan. 2008.

\end{thebibliography}
	\bibliographystyle{abbrv}
\end{document}